\documentclass[11pt,letterpaper]{amsart}
\usepackage{amsmath,amsthm,amssymb,bbm,enumerate,mathrsfs,mathtools}
\usepackage{fullpage}
\usepackage{tikz,xcolor,verbatim}
\usetikzlibrary{calc,shapes}
\usetikzlibrary{arrows}
\usetikzlibrary{hobby}
\usepackage{hyperref}
\usepackage{makecell}
\usepackage[numbers,sort&compress]{natbib}
\usepackage{xspace}
\usepackage{longtable}
\usepackage{dsfont}

\usepackage[normalem]{ulem}

\tikzset{
  bigblue/.style={circle, draw=blue!80,fill=blue!40,thick, inner sep=1.5pt, minimum size=5mm},
  bigred/.style={circle, draw=red!80,fill=red!40,thick, inner sep=1.5pt, minimum size=5mm},
  bigblack/.style={circle, draw=black!100,fill=black!40,thick, inner sep=1.5pt, minimum size=5mm},
  bluevertex/.style={circle, draw=blue!100,fill=blue!100,thick, inner sep=0pt, minimum size=2mm},
  redvertex/.style={circle, draw=red!100,fill=red!100,thick, inner sep=0pt, minimum size=2mm},
  blackvertex/.style={circle, draw=black!100,fill=black!100,thick, inner sep=0pt, minimum size=2mm},  
  whitevertex/.style={circle, draw=black!100,fill=white!100,thick, inner sep=0pt, minimum size=2mm},  
  smallblack/.style={circle, draw=black!100,fill=black!100,thick, inner sep=0pt, minimum size=1.3mm},
    microvert/.style={circle, draw=black!100,fill=black!100, inner sep=0pt, minimum size=0.0001mm},
  smallwhite/.style={circle, draw=black!100,fill=white!100,thick, inner sep=0pt, minimum size=1mm},
  empty/.style={draw=none, fill=none}  
}

\makeatletter
\pgfdeclareshape{myNode}{
  \inheritsavedanchors[from=rectangle] 
  \inheritanchorborder[from=rectangle]
  \inheritanchor[from=rectangle]{center}
  \inheritanchor[from=rectangle]{north}
  \inheritanchor[from=rectangle]{south}
  \inheritanchor[from=rectangle]{west}
  \inheritanchor[from=rectangle]{east}
  \backgroundpath{
    \southwest \pgf@xa=\pgf@x \pgf@ya=\pgf@y
    \northeast \pgf@xb=\pgf@x \pgf@yb=\pgf@y
    \pgfsetcornersarced{\pgfpoint{5pt}{5pt}}
    \pgfpathmoveto{\pgfpoint{\pgf@xa}{\pgf@ya}}
    \pgfpathlineto{\pgfpoint{\pgf@xa}{\pgf@yb}}
    \pgfpathlineto{\pgfpoint{\pgf@xb}{\pgf@yb}}
    \pgfsetcornersarced{\pgfpoint{5pt}{5pt}}
    \pgfpathlineto{\pgfpoint{\pgf@xb}{\pgf@ya}}
    \pgfpathclose
 }
}
\makeatother

\title{Reconfiguring Graph Homomorphisms on the Sphere}

\author{Jae-Baek Lee}
\address[Jae-Baek Lee and Mark Siggers]{College of Natural Sciences, Kyungpook National University, Daegu 702-701, South Korea}
\email{dlwoqor0923@gmail.com, mhsiggers@knu.ac.kr}

\author{Jonathan A. Noel}
\address[Jonathan A. Noel]{Department of Computer Science and DIMAP, University of Warwick, Coventry CV4 7AL.}  
\curraddr{Mathematics Institute and DIMAP, University of Warwick, Coventry CV4 7AL.} 
\email{j.noel@warwick.ac.uk}

\author{Mark Siggers}
              
\thanks{
The first author is supported by the Kyungpook University BK21 Grant. 
This work was initiated while the second author was visiting Kyungpook National University; he would like to thank the first and third authors and the university for their hospitality.
The third author is supported by Korean NRF Basic Science Research Program (2015-R1D1A1A01057653) funded by the Korean government (MEST) and the Kyungpook National University Research Fund. 
}

\date{}

\newtheorem{thm}[equation]{Theorem}
\newtheorem{lem}[equation]{Lemma}

\newtheorem{fact}[equation]{Fact}

\theoremstyle{definition}
\newtheorem{defn}[equation]{Definition}
\newtheorem{obs}[equation]{Observation}

\newtheorem{step}{Step}
\newtheorem*{ack}{Acknowledgements}

\theoremstyle{remark}
\newtheorem{rem}[equation]{Remark}

\newtheoremstyle{case}{}{}{\normalfont}{}{\itshape}{\normalfont:}{ }{}

\theoremstyle{case}

\numberwithin{equation}{section}

\newcommand\Hrec[1]{\textsc{$#1$-Recolouring}}

\newcommand{\Hom}{\operatorname{Hom}}
\newcommand{\Ret}{\operatorname{Ret}}

\newcommand{\bHom}{\operatorname{\mathbf{Hom}}}
\renewcommand\P{\mathcal{P}}

\begin{document}

\begin{abstract}
Given a loop-free graph $H$, the reconfiguration problem for homomorphisms to $H$ (also called $H$-colourings) asks: given two $H$-colourings $f$ of $g$ of a graph $G$, is it possible to transform $f$ into $g$ by a sequence of single-vertex colour changes such that every intermediate mapping is an $H$-colouring? This problem is known to be polynomial-time solvable for a wide variety of graphs $H$ (e.g. all $C_4$-free graphs) but only a handful of hard cases are known. We prove that this problem is PSPACE-complete whenever $H$ is a $K_{2,3}$-free quadrangulation of the $2$-sphere (equivalently, the plane) which is not a $4$-cycle. From this result, we deduce an analogous statement for non-bipartite $K_{2,3}$-free quadrangulations of the projective plane. This include several interesting classes of graphs, such as odd wheels, for which the complexity was known, and $4$-chromatic generalized Mycielski graphs, for which it was not.

If we instead consider graphs $G$ and $H$ with loops on every vertex (i.e. reflexive graphs), then the reconfiguration problem is defined in a similar way except that a vertex can only change its colour to a neighbour of its current colour. In this setting, we use similar ideas to show that the reconfiguration problem for $H$-colourings is PSPACE-complete whenever $H$ is a reflexive $K_{4}$-free triangulation of the $2$-sphere which is not a reflexive triangle. This proof applies more generally to reflexive graphs which, roughly speaking, resemble a triangulation locally around a particular vertex. This provides the first graphs for which \Hrec{H} is known to be PSPACE-complete for reflexive instances.
\end{abstract}
\maketitle

\vspace{-1em}

\section{Introduction}

All graphs in this paper are assumed to be finite, undirected and without multiple edges, unless otherwise specified. A vertex is said to be \emph{reflexive} if it has a loop and \emph{irreflexive} otherwise. A graph is said to be \emph{reflexive} if all of its vertices are reflexive and \emph{irreflexive} if all of its vertices are irreflexive. A \emph{homomorphism} from a graph $G$ to a graph $H$ is a mapping $f:V(G)\to V(H)$ such that $f(u)f(v)\in E(H)$ for every $uv\in E(G)$. For the sake of brevity, a homomorphism $f$ from $G$ to $H$ will sometimes be referred to as an \emph{$H$-colouring} of $G$ and, for $v\in V(G)$, we call $f(v)$ the \emph{colour} of $v$. We denote the set of all $H$-colourings of a graph $G$ by $\Hom(G,H)$. 

Given an irreflexive graph $H$ and two $H$-colourings $f$ and $g$ of a graph $G$, a \emph{reconfiguration sequence} taking $f$ to $g$ is a sequence $f_0,\dots,f_m\in \Hom(G,H)$ such that  $f_0=f$, $f_m=g$ and $f_i$ differs from $f_{i+1}$ on a unique vertex for $0\leq i\leq m-1$. If there exists a reconfiguration sequence taking $f$ to $g$, then we say that $f$ \emph{reconfigures} to $g$.  We are interested in the complexity of the following decision problem, called \Hrec{H}:
\begin{itemize}
\item[] \textbf{Instance:} A graph $G$ and $f,g\in \Hom(G,H)$.
\item[] \textbf{Question:} Does $f$ reconfigure to $g$?
\end{itemize}
The \Hrec{H} problem is part of a growing area known as ``combinatorial reconfiguration,'' a central focus of which is to determine the complexity of deciding whether a given solution to a combinatorial problem can be transformed into another by applying a sequence of allowed modifications. For further background on combinatorial reconfiguration in general, see~\cite{Demaine,Brooks,SAT,Mohar,KempeFeg,indepCog,Moritz,shortest} and the surveys of van den Heuvel~\cite{JanSurvey}, Ito and Suzuki~\cite{webSurvey} and Nishimura~\cite{Nishimura}.

An interesting special case of the \Hrec{H} problem is when $H$ is a complete graph on $k$ vertices, in which case $H$-colourings are nothing more than proper $k$-colourings. Cereceda, van den Heuvel and Johnson~\cite{3col} showed that  \Hrec{K_3} can be solved in time $O(|V(G)|^2)$. This came as some surprise, given that it is NP-complete to decide whether a graph admits a $K_3$-colouring.  On the other hand, Bonsma and Cereceda~\cite{4col} showed that the complexity jumps drastically for larger cliques: for every fixed $k\geq4$, the \Hrec{K_k} problem is PSPACE-complete.\footnote{Note that it is not hard to see that \Hrec{H} is in PSPACE for every finite graph $H$.} Later, Brewster, McGuinness, Moore and Noel~\cite{circular} extended this dichotomy to the case when $H$ is a ``circular clique.'' 

Wrochna~\cite{Wrochna} developed ideas inspired by algebraic topology to prove the remarkably general result that \Hrec{H} is solvable in polynomial time whenever $H$ does not contain a cycle of length $4$. By further refining his topological approach, Wrochna~\cite{WrochnaHed} (see also~\cite{WrochnaPhD}) proved a ``multiplicativity'' result for graphs without cycles of length $4$ which is closely connected to Hedetniemi's Conjecture~\cite{Hedet}; very recently, Tardif and Wrochna~\cite{TW} have extended these methods beyond the setting of $C_4$-free graphs. 

On the hardness side, however, only a few examples are known. As we have mentioned, the results of~\cite{circular,4col} show that the problem is PSPACE-complete for certain cliques and circular cliques. In addition, Wrochna~\cite{WrochnaMasters} proved that there exists a graph $H$ such that \Hrec{H} is PSPACE-complete even when the instance graph $G$ is just a cycle and Brewster, Lee, Moore, Noel and Siggers~\cite{wheelFrozen} proved that \Hrec{H} is PSPACE-complete if $H$ is an odd wheel; for $k\geq3$, the \emph{wheel} $W_k$ is the graph consisting of an irreflexive cycle of length $k$ and a vertex adjacent to every vertex of the cycle and it is \emph{odd} if $k$ is odd. 

Our goal in this paper is to obtain a rich class of graphs $H$ for which \Hrec{H} is PSPACE-complete. Throughout the paper, a \emph{quadrangulation} is a connected irreflexive graph admitting an embedding in the $2$-sphere (or, equivalently, the plane) in which every face is bounded by four edges. Our main result for irreflexive graphs is the following. 

\begin{thm}
\label{quadThm}
If $H$ is a finite irreflexive quadrangulation not containing $K_{2,3}$ as a subgraph and not isomorphic to the $4$-cycle, then \Hrec{H} is PSPACE-complete. 
\end{thm}

To prove Theorem~\ref{quadThm}, we will reduce \Hrec{K_4} to \Hrec{H} and apply the result of Bonsma and Cereceda~\cite{4col} mentioned above. However it will be apparent in the proof that essentially the same approach could have been used to reduce \Hrec{F} to \Hrec{H} for any graph $F$ and so there is nothing particularly special about the choice of $K_4$ (except that \Hrec{K_4} is known to be PSPACE-complete). 

As an application of Theorem~\ref{quadThm}, we will derive an analogous statement for non-bipartite quadrangulations of the projective plane. We are indebted to an anonymous referee who pointed out that one of our original proofs (which was only stated for odd wheels) holds in this generality.
  
\begin{thm}
\label{ppCor}
If $H$ is a non-bipartite quadrangulation of the projective plane not containing $K_{2,3}$ as a subgraph, then \Hrec{H} is PSPACE-complete. Moreover, it remains PSPACE-complete when restricted to instances $(G,f,g)$ such that $G$ is bipartite.   
\end{thm}

This applies to interesting families of graphs such as odd wheels and $4$-chromatic generalized Mycielski graphs. Therefore, it recovers, and vastly extends, the result of  Brewster et al.~\cite{wheelFrozen} that \Hrec{H} is PSPACE-complete for any odd wheel $H$. The key to the proof of Corollary~\ref{ppCor} is that the ``bipartite double cover'' of a non-bipartite quadrangulation of the projective plane is a quadrangulation of the sphere. If $H$ is an even wheel, then the bipartite double cover is no longer a quadrangulation of the sphere itself, but it can be ``retracted'' to one. Using this, we obtain the following explicit extension of the result for odd cycles from \cite{wheelFrozen}.

\begin{thm}
\label{wheelThm}
For $k \geq 3$ and $k\neq 4$, \Hrec{W_k} is PSPACE-complete.  Moreover, it remains PSPACE-complete when restricted to instances $(G,f,g)$ such that $G$ is bipartite.  
\end{thm}

Note that the condition that $k\neq4$ in Theorem~\ref{wheelThm} is necessary unless $\text{P}=\text{PSPACE}$. To see this, observe that the graph $W_4$ contains two pairs of vertices with identical neighbourhoods and identifying both of these pairs ``folds'' $W_4$ to $K_3$. This observation can be used to show that \Hrec{W_4} and \Hrec{K_3} are polynomially equivalent and so the former is solvable in polynomial time by the result of~\cite{3col}. For details on how these ``folding reductions'' work in general, see Wrochna~\cite[Proposition~4.3]{WrochnaMasters}.

Theorem~\ref{quadThm} fits a general theme, which first emerged in the topological approach of Wrochna~\cite{Wrochna}, that the complexity of \Hrec{H} may be closely related to the structure of a topological complex in which vertices, edges and $4$-cycles of $H$ (and, more generally, complete bipartite subgraphs) are faces. That is, Wrochna's result~\cite{Wrochna} says that, if this complex is ``thin'' in the sense that all of its faces are $0$- or $1$-dimensional, then its simple topological structure can be exploited to obtain a polynomial-time algorithm, whereas our  result  says that if this complex has basically the same topology as a $2$-sphere, then the problem is PSPACE-complete. We further discuss the likely connections between \Hrec{H} and the ``topology'' of $H$ in the setting of reflexive graphs $H$, which we discuss next, as the connections are more natural in this setting.

Consider now the reconfiguration problem for $H$-colourings of $G$ where both of the graphs $G$ and $H$ are reflexive, studied previously in~\cite{digraphs}. In this setting, the definition of a reconfiguration sequence is somewhat different; it is defined to be a sequence $f_0,\dots,f_m\in \Hom(G,H)$ such that $f_i$ and $f_{i+1}$ differ on a unique vertex $u_i$ for $0\leq i\leq m-1$ and $f_{i+1}\left(u_i\right)$ is a neighbour of $f_i\left(u_i\right)$. That is, it is the same as the definition for irreflexive graphs, but with an additional restriction that a vertex can only change its colour to a neighbour of its current colour. 

Let us justify this extra condition by redefining the notion of homomorphism  reconfiguration in terms of paths in the well known Hom-graph. For general graphs $G$ and $H$, the \emph{Hom-graph}, denoted $\bHom(G,H)$, is the graph with vertex set $\Hom(G,H)$ in which two homomorphisms $\phi$ and $\psi$ in $\Hom(G,H)$ are adjacent if $\phi(x) \psi(y)\in E(H)$ for every $xy\in E(G)$. The Hom-graph can be viewed as the $2$-skeleton of the Hom-complex which was first introduced by Lov\'{a}sz~\cite{LovaszKneser} in his celebrated proof of the Kneser Conjecture. 

When $G$ and $H$ are irreflexive, it is easily shown that one can reconfigure between two homomorphisms $\phi$ and $\psi$ from $G$ to $H$ if and only if they are in the same component of $\bHom(G,H)$; see, e.g.,~\cite[Proposition 3.2]{BrewNoel} for a proof. However, for reflexive graphs, the presence of a loop on every vertex of $G$ means that, for any two homomorphisms $\phi$ and $\psi$ which are adjacent in $\bHom(G,H)$ and vertex $v\in V(G)$, we must have that $\phi(v)$ and $\psi(v)$ are neighbours in $H$. Using this observation, it is not hard to see that, when $G$ and $H$ are reflexive, there is a path between two $H$-colourings of $G$ if and only if one can be reconfigured to the other with the extra condition mentioned above. This definition in terms of the Hom-graph allows us to define the reconfiguration not only for reflexive graphs, but also for graphs in which loops are allowed, but not required.  Further, defining it this way 
helps preserve connections between homomorphism reconfiguraton and important concepts in graph homomorphism theory, such as the connection to the Hom-complex. 

Basic among these concepts is the fact from Brightwell and Winkler~\cite{BW} that a graph $H$ is ``dismantlable'' if and only if $\bHom(G,H)$ is connected for all $H$.  From this we get that \Hrec{H} is trivial if $H$ is dismantlable.   The  notion of ``folding'' mentioned above can be seen as a irreflexive version of dismantling. Using the same proof as for folding, it is easily shown that if $H$ dismantles to $H'$, then the problems $\Hrec{H}$ and $\Hrec{H'}$ are polynomially equivalent.

We consider such properties to be topological properties as there is a  close connection between $\bHom(G,H)$ being disconnected for various $G$ and non-trivial homotopy in the \emph{clique complex} of $H$---the simplicial complex on the vertices of $H$ whose $k$-simplices are the reflexive $(k+1)$-cliques of $H$.  While $H$ being dismantlable implies that its clique complex deformation retracts to a single vertex, and so has trivial homotopy,  non-trivial homomotopy in the clique complex of $H$ is the main source of disconnectedness in $\bHom(C,H)$ for various cycles $C$ or higher dimensional analogues.  

In \cite{La}, Larose, shows that any non-trivial homotopy in the clique complex of a reflexive graph $H$ yields an NP-complete retraction problem $\Ret(H)$ (which is the usual analogue of the homomorpism problem for reflexive graphs).  Wrochna's result suggests that non-trivial homotopy of dimension $1$ will not suffice to make a hard reconfiguration problem, and indeed, in a forthcoming paper~\cite{LNS2}, we obtain an analogue of Wrochna's result in the reflexive setting by showing that \Hrec{H} is polynomial time solvable if $H$ contains no triangles---that is, if the clique complex has only $0$- and $1$-simplices. To find $H$ for which \Hrec{H} is not polynomial time solvable, it seems that the natural candidates are graphs $H$ with $2$-dimensional holes.  

We believe that any two dimensional hole in the clique complex of a reflexive graph $H$ yields a hard reconfiguration problem. A reflexive graph is called a \emph{triangulation} if the underlying irreflexive graph is connected and can be embedded in the plane so that all faces are bounded by three edges.  Triangulations of a sphere that are not $K_4$ are the simplest examples of simplicial complexes with two dimensional holes. To make an explicit conjecture, we expect for a reflexive graph $H$ that if $\bHom(S,H)$ is disconnected for any reflexive triangulation $S$ of a sphere, then \Hrec{H} is PSPACE-complete.

The following result, our main result for reflexive graphs,  is a step towards resolving this, and to our knowledge,  provides the first examples of graphs for which \Hrec{H} is PSPACE-complete when restricted to reflexive instances.

\begin{thm}
\label{triThm}
If $H$ is a finite reflexive triangulation not containing $K_4$ as a subgraph and not isomorphic to a reflexive triangle, then \Hrec{H} is PSPACE-complete when restricted to instances $(G,f,g)$ such that $G$ is reflexive. 
\end{thm}

As is the case in \cite{La} we do not expect the converse of our conjecture to hold.  Indeed, it does not. 
We will prove a more general result than Theorem \ref{triThm} which only requires $H$ to have the ``local'' structure of a triangulation near a particular vertex, and the ``global'' property of  ``stiffness'' which effectivly stops this local structure from dismantling to the particular vertex.  This more general result is stated and proved in Section~\ref{triSection} (Theorem~\ref{localThm}).


The rest of the paper is organized as follows. In the next section, we give an overview of the proofs of Theorems~\ref{quadThm} and~\ref{triThm}. In particular, we describe the types of gadgets used in the proofs and show that the existence of such gadgets is sufficient to prove the main theorems. In Section~\ref{quadSection}, we build up several basic structural properties about graphs satisfying the hypotheses of Theorem~\ref{quadThm} and use them to construct the required gadgets. We then deduce Theorem~\ref{ppCor} in Section~\ref{wheelSection} and the using a reduction for even wheels, apply it to prove Theorem~\ref{wheelThm}. 
 In Section~\ref{triSection}, we state and prove a generalization of Theorem~\ref{triThm}. 


\section{Overview of the Main Reduction}
\label{overviewSection}

The goal of this section is to introduce the main  gadgets used in the proofs of Theorems~\ref{quadThm} and~\ref{triThm}. We will mainly focus on Theorem~\ref{quadThm}, only commenting briefly at the end of the section about how the ideas can be adapted to the setting of Theorem~\ref{triThm}. For the time being, we let $H$ be any $K_{2,3}$-free finite quadrangulation other than the $4$-cycle. We claim that $H$ must contain a vertex of degree three. Indeed, by Euler's Polyhedral Formula, it must contain a vertex of degree at most three and since it is a quadrangulation which is $K_{2,3}$-free and not isomorphic to $C_4$, it cannot have a vertex of degree less than three (see Lemma~\ref{minDegis3} for a formal proof). So, we may choose an arbitrary vertex of degree three and label it $0$. Also, label the three faces incident with $0$ by $f_1,f_2$ and $f_3$ and, for $i\in\{1,2,3\}$, label the unique vertex incident to $f_i$ and not adjacent to $0$ by $i$. Since the three neighbours of $0$ are distinct and $H$ is $K_{2,3}$-free, the vertices $1,2$ and $3$ must be distinct. We let $\alpha_{1,2},\alpha_{2,3}$ and $\alpha_{3,1}$ be defined so that $\alpha_{i,j}$ is the common neighbour of vertices $i,j$ and $0$, which is unique because $H$ is $K_{2,3}$-free; see Figure~\ref{around0}. 

\begin{figure}[htbp]
\begin{center}
\begin{tikzpicture}
   \newdimen\R
   \R=1.65cm
   \newdimen\smallR
   \smallR=1.0cm
   \draw
\foreach [count=\n from 1] \x in {150,270,...,390} {
	(\x:\R) node [smallblack, label={[label distance=0pt]\x:{\n}}] (v\n){} 
};

\draw\foreach [count=\n from 1] \x in {150,270,...,390} {
	(\x:\smallR) node (f\n){$f_{\n}$} 
};

\draw (90:\R) node [smallblack, label={[label distance=0pt]90:{$\alpha_{3,1}$}}] (alpha31){};

\draw (210:\R) node [smallblack, label={[label distance=0pt]210:{$\alpha_{1,2}$}}] (alpha12){};

\draw (330:\R) node [smallblack, label={[label distance=0pt]330:{$\alpha_{2,3}$}}] (alpha23){};

\draw (0,0) node [smallblack, label={[label distance=0pt]160:{$0$}}] (v0) {};

\draw (alpha31)--(v1)--(alpha12)--(v2)--(alpha23)--(v3)--(alpha31);
\draw(v0)--(alpha31);
\draw(v0)--(alpha12);
\draw(v0)--(alpha23);
\end{tikzpicture}
\end{center}

\caption{The local structure near vertex $0$.}
\label{around0}
\end{figure}
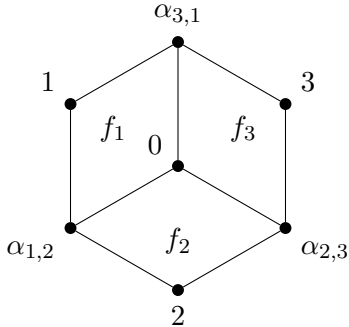

Our aim is to reduce the \Hrec{K_4} problem, which was shown to be PSPACE-complete in~\cite{4col}, to the \Hrec{H} problem. To this end, we let $(G,f,g)$ be an instance of \Hrec{K_4}, where $V(K_4)=\{1,2,3,4\}$ and will construct an instance $\left(G',f',g'\right)$ of \Hrec{H} such that $|V(G')|=O\left(|V(G)|^{2}\right)$ and $f$ reconfigures to $g$ if and only if $f'$ reconfigures to $g'$. The construction is broken down into four steps. We describe the first two steps now and postpone the description of the third and fourth until after some additional discussion. 

\begin{step}
\label{step1}
Each vertex $u\in V(G)$ is represented by four vertices $u_1,u_2,u_3$ and $u_4$ in $G'$. 
\end{step}

\begin{step}
\label{step2}
For each $u\in V(G)$ and $i\in\{1,2,3,4\}$, define 
\[f'\left(u_i\right):=\begin{cases} 1 & \text{if }f(u)=i,\\
0 & \text{otherwise}
\end{cases}\]
and define $g'\left(u_i\right)$ analogously.
\end{step}

As one may be able to glean from Step~\ref{step2}, the images of the vertices $u_1,u_2,u_3,u_4$ under an $H$-colouring of $G'$ will be used to encode the colour of $u$ under an associated $K_4$-colouring in a simple way. That is, we think of $u_i$ mapping to 1/0 as meaning that the $i$th colour is ``turned on/turned off'' at the vertex $u$. Note that we will not mind if more than one colour is turned on at $u$ (in fact, it is necessary to allow this in order for $u$ to transition between colours in the associated $K_4$-colourings). What we need in order to make the reduction work is to design gadgets which force the following properties to be maintained throughout any reconfiguration sequence starting with $f'$:
\begin{enumerate}[(i)]
\item\label{01} For each $u\in V(G)$ and $i\in\{1,2,3,4\}$, the colour of $u_i$ is either $0$ or $1$.
\item\label{proper} For each $uv\in E(G)$ and $i\in\{1,2,3,4\}$, the vertices $u_i$ and $v_i$ cannot map to $1$ at the same time.
\item \label{mustPick} For each $u\in V(G)$ at least one of the vertices $u_1,u_2,u_3,u_4$ maps to $1$. 
\end{enumerate}
Given that these properties are maintained, it will easily follow that if $f'$ reconfigures to $g'$, then $f$ reconfigures to $g$. Indeed, for each $H$-colouring in the reconfiguration sequence taking $f'$ to $g'$, we define a $K_4$-colouring by assigning each $u\in V(G)$ to the minimum $i\in\{1,2,3,4\}$ such that $u_i$ is mapped to $1$. By (\ref{mustPick}), such an $i$ always exists and, by (\ref{proper}), this choice will always produce a $K_4$-colouring. Clearly, any two consecutive such $K_4$-colourings will differ on at most one vertex. Also, applying this transformation to $f'$ itself yields $f$, and applying it to $g'$ yields $g$. Therefore, we obtain a reconfiguration sequence taking $f$ to $g$, as desired. 

However, when trying to prove the other direction, one soon realizes that it is important to make the gadgets sufficiently ``flexible'' so that we can mimic any reconfiguration sequence taking $f$ to $g$ by a reconfiguration sequence of $H$-colourings taking $f'$ to $g'$. That is, we need to not only block the ``undesirable configurations'' (e.g. $u_2$ and $v_2$ both mapping to $1$ for $uv\in E(G)$), but also to allow any sorts of ``allowed transitions'' between configurations (e.g. changing the colour of $u_3$ from $1$ to $0$ while $u_4$ is coloured with $1$). This discussion is an attempt to motivate condition (\ref{patternTransition}) of the following technical-looking definition. 
For a function $f$ on a set $X$ and a vector $x = (x_1, \dots, x_k)\in X^k$, we write $f(x)$ for $(f(x_1), \dots, f(x_k))$.  

\begin{defn}
\label{patternDef}
Let $H$ be a graph, $k$ be a positive integer and $\P \subseteq V(H)^k$ be a $k$-ary relation on $V(H)$ (the elements of which we call \emph{$k$-patterns}). A \emph{$\P$-gadget} is a graph $Y  = Y(x)$ where $x = (x_1, \dots, x_k)$ is any ordered set of {\em signal} vertices $x_1,\dots,x_k\in V(Y)$ such that the following hold.
\begin{enumerate}[(a)]
\item\label{patterns} For each $p \in \P$ there is a canonical $\zeta_p \in \Hom(Y,H)$ such that $\zeta_p(x) = p$.
\item\label{stuck} If $\psi\in \Hom(Y,H)$ reconfigures to $\zeta_p$ for some $p\in\P$, then $\psi(x)\in\P$. 
\item\label{patternTransition} If $p,q\in \P$ differ on at most one coordinate, then there exists a reconfiguration sequence taking $\zeta_p$ to $\zeta_q$ such that $\xi(x)\in\left\{p,q\right\}$ for every element $\xi$ of this sequence. 
\end{enumerate}
\end{defn}

As we will see shortly, the proof of Theorem~\ref{quadThm} boils down to establishing the following two lemmas. Given distinct vertices $a,b$ of $H$, say that $a$ is \emph{across} from $b$ if $a$ is not adjacent to $b$ and there exists a face of $H$ incident to both $a$ and $b$. 

\begin{lem}
\label{nboLem}
Let $H$ be a finite $K_{2,3}$-free quadrangulation which is not a $4$-cycle. If $a_0$ is across from $a_1$ and $b_0$ is across from $b_1$, then there there exists an $\left\{(a_0,b_0),(a_1,b_0),(a_0,b_1)\right\}$-gadget.
\end{lem}

\begin{lem}
\label{nazLem}
Let $H$ be a finite $K_{2,3}$-free quadrangulation which is not a $4$-cycle and let $0,1\in V(H)$ such that $0$ has degree three and $1$ is across from $0$. Then there exists a $\left(\{0,1\}^4\setminus\{(0,0,0,0)\}\right)$-gadget.
\end{lem}

In general, for a graph $H$ and distinct  $0,1\in V(H)$, we refer to a  $\left\{(0,0),(1,0),(0,1)\right\}$-gadget as a \emph{not-both-one gadget} and a $\left(\{0,1\}^4\setminus\{(0,0,0,0)\}\right)$-gadget as a \emph{not-all-zero gadget}. Given these lemmas, we complete the construction of $(G',f',g')$ by applying the following steps. See Figure~\ref{G'Fig} for an illustration of the full construction of $(G',f',g')$. 

\begin{step}
\label{step3}
For each ordered pair $(u,v)$ with $uv\in E(G)$ and $i\in\{1,2,3,4\}$, we add a not-both-one gadget $Y(u_i,v_i)$ to $G'$ which is disjoint from all vertices added so far except for the signal vertices $u_i$ and $v_i$. We define $f'$ on $Y$ to agree with $\zeta_{\left(f'(u_i),f'(v_i)\right)}$. The definition of $g'$ on $Y$ is analogous.\footnote{Technically, we only require one of the gadgets $Y(u_i,v_i)$ or $Y(v_i,u_i)$ in order for the reduction to work. However, adding both gadgets provides symmetry which is convenient in the exposition of the proof of Lemma~\ref{gadgetImpliesHard}.}
\end{step}

\begin{step}
\label{step4}
For each $u\in V(G)$ we add a not-all-zero gadget $Z(u_1,u_2,u_3,u_4)$ to $G'$, disjoint from all vertices added so far except for the signal vertices $u_1,u_2,u_3,u_4$. We define $f'$ on $Z$ to agree with $\zeta_{\left(f'(u_1),f'(u_2),f'(u_3),f'(u_4)\right)}$. The definition of $g'$ on $Z$ is analogous.
\end{step}

This completes the construction of the instance $(G',f',g')$ of \Hrec{H}. Next, we prove a general lemma which says that the existence of a not-both-one gadget and a not-all-zero gadget is enough to prove that \Hrec{H} is PSPACE-complete.

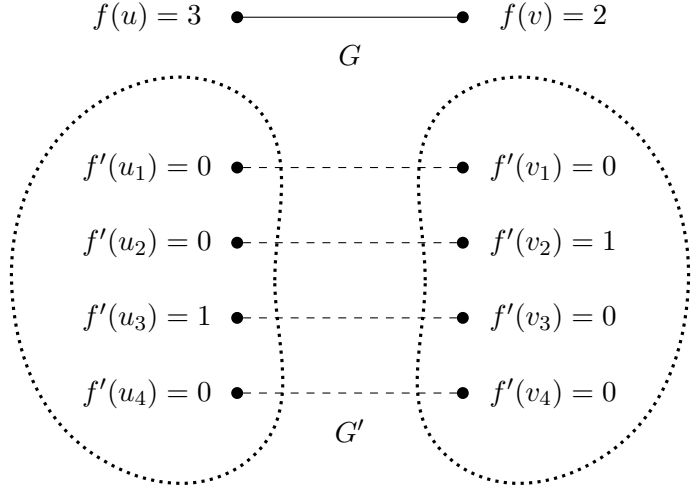
\begin{figure}[htbp]
\begin{center}
\begin{tikzpicture}
\draw (-1.5,0) node [smallblack,label={[label distance=4pt]90:{}}] (u) {}
--(1.5,0) node [smallblack,label={[label distance=4pt]90:{}}] (v) {};
\draw (-2.7,-0.0) node (fu) {$f(u)=3$};
\draw (2.7,-0.0) node (fv) {$f(v)=2$};

\draw (0,-0.5) node (G) {$G$};

\begin{scope}[shift={(0,-2)}]

\draw (-1.5,0) node [smallblack,label={[label distance=4pt]90:{}}] (u1) {};
\draw(1.5,0) node [smallblack,label={[label distance=4pt]90:{}}] (v1) {};
\draw[dashed](u1)--(v1);

\draw (-2.7,-0.0) node (fu1) {$f'(u_1)=0$};
\draw (2.7,-0.0) node (fv1) {$f'(v_1)=0$};

\begin{scope}[shift={(0,-1)}]
\draw (-1.5,0) node [smallblack,label={[label distance=4pt]90:{}}] (u2) {};
\draw(1.5,0) node [smallblack,label={[label distance=4pt]90:{}}] (v2) {};
\draw[dashed](u2)--(v2);

\draw (-2.7,-0.0) node (fu2) {$f'(u_2)=0$};
\draw (2.7,-0.0) node (fv2) {$f'(v_2)=1$};
\end{scope}

\begin{scope}[shift={(0,-2)}]
\draw (-1.5,0) node [smallblack,label={[label distance=4pt]90:{}}] (u3) {};
\draw(1.5,0) node [smallblack,label={[label distance=4pt]90:{}}] (v3) {};
\draw[dashed](u3)--(v3);

\draw (-2.7,-0.0) node (fu3) {$f'(u_3)=1$};
\draw (2.7,-0.0) node (fv3) {$f'(v_3)=0$};
\end{scope}

\begin{scope}[shift={(0,-3)}]
\draw (-1.5,0) node [smallblack,label={[label distance=4pt]90:{}}] (u4) {};
\draw(1.5,0) node [smallblack,label={[label distance=4pt]90:{}}] (v4) {};
\draw[dashed](u4)--(v4);

\draw (-2.7,-0.0) node (fu4) {$f'(u_4)=0$};
\draw (2.7,-0.0) node (fv4) {$f'(v_4)=0$};

\draw (0,-0.5) node (G) {$G'$};
\end{scope}

\draw[very thick, dotted] (-1,-1.5) to [closed, quick curve through={(-1,-1.5) . . (-1.2,0.75) . . (-1.5,1) . . (-3.5,0.75) . . (-4.5,-1.5) . . (-3.5,-3.75) . . (-1.5,-4) . . (-1.2,-3.75) . . (-1,-1.5)}] (-1,-1.5);

\draw[very thick, dotted] (1,-1.5) to [closed, quick curve through={(1,-1.5) . . (1.2,0.75) . . (1.5,1) . . (3.5,0.75) . . (4.5,-1.5) . . (3.5,-3.75) . . (1.5,-4) . . (1.2,-3.75) . . (1,-1.5)}] (1,-1.5);
\end{scope}
\end{tikzpicture}
\end{center}
\caption{An illustration of the way in which two adjacent vertices $u,v$ of $G$ are represented in the graph $G'$. Each dashed line connects signal vertices of a pair of not-both-one gadgets and each thick dotted curve encloses the four signal vertices of a not-all-zero gadget.}
\label{G'Fig}
\end{figure}

\begin{lem}
\label{gadgetImpliesHard}
If $H$ is a finite graph and $0,1$ are distinct vertices of $H$ such that there exists a not-both-one gadget and a not-all-zero gadget, then \Hrec{H} is PSPACE-complete. 
\end{lem}

\begin{proof}
Given an instance $(G,f,g)$ of the \Hrec{K_4} problem, we let  $(G',f',g')$ be an instance for \Hrec{H} constructed using Steps~\ref{step1}-\ref{step4} outlined above. Note that the size of these gadgets depends only on $H$ and so $|V(G')|=O\left(|V(G)|+|E(G)|\right) =O\left(|V(G)|^2\right)$. We show that $f$ reconfigures to $g$ if and only if $f'$ reconfigures to $g'$. 

First suppose that there is a reconfiguration sequence $f_0',\dots,f_m'\in \Hom(G',H)$ taking $f'$ to $g'$. By Step~\ref{step4} and the definition of the not-all-zero gadget, we know that, for each $u\in V(G)$ and $0\leq j\leq m$, there exists $i\in\{1,2,3,4\}$ such that $f_j'(u_i)=1$. For $1\leq j\leq m$, define $f_j:V(G)\to\{1,2,3,4\}$ by
\[f_j(u):=\min\left\{i: f_j'(u_i)=1\right\}\]
for $u\in V(G)$. By Step~\ref{step3} and the definition of the not-both-one gadget, we have that each $f_j$ is a $K_4$-colouring of $G$; that is, adjacent vertices of $G$ receive distinct colours. Clearly, by construction, $f_0=f$ and $f_m=g$. Also, since, for $0\leq j\leq m-1$, the mappings $f'_j$ and $f'_{j+1}$ differ on exactly one vertex, the mappings $f_j$ and $f_{j+1}$ differ on at most one vertex. Thus, we can take a subsequence of $f_0,\dots,f_m$ obtained by deleting repetitions (i.e. removing one of $f_i$ or $f_{i+1}$ when $f_i=f_{i+1}$) to get a reconfiguration sequence taking $f$ to $g$. 

For the other direction, suppose that there is a reconfiguration sequence $f_0,\dots,f_m\in \Hom(G,K_4)$ taking $f$ to $g$. We construct a reconfiguration sequence taking $f'$ to $g'$. We may assume that $m=1$, i.e. that $f$ and $g$ differ on a unique vertex $u$, since the general case follows by induction on $m$. So, without loss of generality, we assume that $f(u)=1$ and that $g(u)=2$ and that $f(v)=g(v)$ for all $v\in V(G)\setminus\{u\}$. 

Since $f$ and $g$ are $K_4$-colourings which differ only on $u$, we have that $f(v),g(v)\in\{3,4\}$ for every vertex $v$ adjacent to $u$. By construction, this means that $f'$ assigns
\begin{itemize}
\item the colouring $\zeta_{(1,0)}$ to $Y(u_1,v_1)$,
\item the colouring $\zeta_{(0,1)}$ to $Y(v_1,u_1)$, 
\item the colouring $\zeta_{(0,0)}$ to both $Y(u_2,v_2)$ and $Y(v_2,u_2)$, and
\item the colouring $\zeta_{(1,0,0,0)}$ to $Y(u_1,u_2,u_3,u_4)$. 
\end{itemize}
Similarly, $g'$ assigns
\begin{itemize}
\item the colouring $\zeta_{(0,0)}$ to both of $Y(u_1,v_1)$ and $Y(v_1,u_1)$, 
\item the colouring $\zeta_{(1,0)}$ to $Y(u_2,v_2)$, \item the colouring $\zeta_{(0,1)}$ to $Y(v_2,u_2)$, and
\item the colouring $\zeta_{(0,1,0,0)}$ to $Y(u_1,u_2,u_3,u_4)$. 
\end{itemize}
On all other vertices of $G'$, the colourings $f'$ and $g'$ agree with one another. By condition (\ref{patternTransition}) of Definition~\ref{patternDef}, we know that we can reconfigure $\zeta_{(0,0)}$ to $\zeta_{(1,0)}$ in such a way that the colour of the first signal vertex stays in $\{0,1\}$ the second is mapped to $0$ throughout. For each neighbour $v$ of $u$, one at a time, we apply the first part of this reconfiguration sequence on $Y(u_2,v_2)$, stopping just before the first step in which the colour of $u_2$ changes from $0$ to $1$. Similarly, on $Y(v_2,u_2)$, apply the first steps of a reconfiguration sequence from $\zeta_{(0,0)}$ to $\zeta_{(0,1)}$ and, on $Y(u_1,u_2,u_3,u_4)$, apply the first steps of a reconfiguration sequence from $\zeta_{(1,0,0,0)}$ to $\zeta_{(1,1,0,0)}$, in all cases stopping just before the first time the colour of $u_2$ changes from $0$ to $1$. Next, we go through each of these gadgets again, one by one, and continue the reconfiguration sequence, this time stopping at the last step in which the colour of $u_2$ is $0$. Then, go through each gadget one last time to complete the reconfiguration sequence. Note that this procedure maintains an $H$-colouring throughout since any two of these gadgets only intersect on $u_2$ and possibly a vertex $v_2$ which does not change colour. 

Thus, we have arrived at a colouring $h'$ which assigns
\begin{itemize}
\item the colouring $\zeta_{(1,0)}$ to $Y(u_1,v_1)$,
\item the colouring $\zeta_{(0,1)}$ to $Y(v_1,u_1)$, 
\item the colouring $\zeta_{(1,0)}$ to $Y(u_2,v_2)$,
\item the colouring $\zeta_{(0,1)}$ to $Y(v_2,u_2)$, and
\item the colouring $\zeta_{(1,1,0,0)}$ to $Y(u_1,u_2,u_3,u_4)$. 
\end{itemize}
and, on all other vertices of $G'$, agrees with both $f'$ and $g'$. By applying the same steps as above with $g'$ in the place of $f'$ and swapping the roles of colours $1$ and $2$, we see that $g'$ reconfigures to $h'$ as well and so, by symmetry and transitivity of the ``reconfigures to'' relation, $f'$ reconfigures to $g'$. This completes the proof.
\end{proof}

Before closing this section, let us make a few remarks about Theorem~\ref{triThm}. Let $H$ be a finite $K_4$-free reflexive triangulation which is not a reflexive triangle. This time, we will let $0$ be an arbitrary vertex and label the neighbours of $0$ apart from $0$ itself by $1,\dots,k$ in clockwise order with respect to the embedding of $H$, where $k$ is the degree of $0$ (we follow the convention that the loop on $0$ does  not count towards its degree). Note that, as $H$ is a triangulation and not a triangle, the neighbours of $0$ apart from $0$ itself must form a reflexive cycle in $H$; this is proved formally in Lemma~\ref{localImplies}. The following lemmas are the crux of the proof of Theorem~\ref{triThm}.

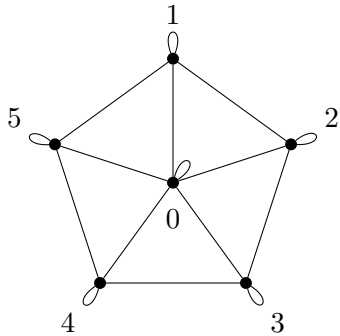
\begin{figure}[htbp]

\begin{tikzpicture}
   \newdimen\R
   \R=1.65cm
  \draw (522:\R)
\foreach [count=\n from 1] \x in {450,378,...,162} {
	-- (\x:\R) node [smallblack, label={[label distance=7pt]\x:{\n}}] (v\n){} 
};
\draw (0,0) node [smallblack, label={[label distance=4pt]270:{$0$}}] (alpha) {};
\draw (alpha) -- (v1);
\draw (alpha) -- (v2);
\draw (alpha) -- (v3);
\draw (alpha) -- (v4);
\draw (alpha) -- (v5);

\path (alpha) edge[ out=75, in=33
                , loop
                , distance=0.4cm]
            node[above=3pt] {} (alpha);

\path (v1) edge[ out=69, in=111
                , loop
                , distance=0.4cm]
            node[above=3pt] {} (v1);
            
\path (v5) edge[ out=142, in=183
                , loop
                , distance=0.4cm]
            node[above=3pt] {} (v5);
            
\path (v4) edge[ out=214, in=255
                , loop
                , distance=0.4cm]
            node[above=3pt] {} (v4);
            
\path (v3) edge[ out=286, in=327
                , loop
                , distance=0.4cm]
            node[above=3pt] {} (v3);
            
\path (v2) edge[ out=358, in=399
                , loop
                , distance=0.4cm]
            node[above=3pt] {} (v2);
\end{tikzpicture}

\caption{The subgraph of $H$ induced by the vertex $0$ and its neighbours in the case $k=5$.}
\label{nb0TriFig}
\end{figure}

\begin{lem}
\label{nboLemTri}
If $H$ be a finite reflexive $K_4$-free triangulation which is not a reflexive triangle and $01\in E(H)$ with $0\neq 1$, then there there exists a not-both-one gadget.
\end{lem}

\begin{lem}
\label{nazLemTri}
If $H$ is a finite reflexive $K_4$-free triangulation which is not a reflexive triangle and $01\in E(H)$ with $0\neq 1$, then there there exists a not-all-zero gadget.
\end{lem}

As was mentioned in the introduction, we will actually prove a more general result than Theorem~\ref{triThm} and, thus, will require more general lemmas than Lemmas~\ref{nboLemTri} and~\ref{nazLemTri}. These lemmas are stated and proved in Section~\ref{triSection}.

\section{Gadgets for Quadrangulations}
\label{quadSection}

Throughout this section, let $H$ be a finite $K_{2,3}$-free quadrangulation which is not a $4$-cycle. We begin by obtaining some basic structural properties of $H$ which will be useful in the proofs of Lemmas~\ref{nboLem} and~\ref{nazLem}.

\subsection{Basic Structural Properties of \texorpdfstring{$\boldsymbol{H}$}{H}}

The following two lemmas highlight some of the main ways in which we exploit the fact that $H$ is $K_{2,3}$-free.

\begin{lem}
\label{nonAdj}
If $abcd$ and $a'b'cd$ are distinct cycles of $H$, then $a$ is not adjacent to $b'$.
\end{lem}

\begin{proof}
Suppose that $a$ is adjacent to $b'$. Then $b,d$ and $b'$ are all common neighbours of $a$ and $c$. As $H$ is $K_{2,3}$-free, $a$ and $c$ cannot have more than two common neighbours, and so it must be the case that some of these vertices coincide. Since $b$ and $d$ are distinct vertices on a cycle (by hypothesis), and so are $b'$ and $d$, it must be the case that $b=b'$. However, now we get that $a,c$ and $a'$ are common neighbours of $b$ and $d$, and these three vertices are distinct from one another because $a$ and $c$ are on a cycle, as are $a'$ and $c$, and the cycles $abcd$ and $a'b'cd$ are distinct by hypothesis. This contradicts the assumption that $H$ is $K_{2,3}$-free and completes the proof. 
\end{proof}

\begin{lem}
\label{minDegis3}
Every vertex of $H$ has degree at least three.
\end{lem}

\begin{proof}
Suppose not. As $H$ is a quadrangulation, it is clear that it has no vertex of degree zero or one. So, let $y$ be a vertex of degree two and let $x$ and $z$ be its two neighbours. Since $H$ is a quadrangulation and $y$ has degree two, the faces $f_1$ and $f_2$ incident to the edge $xy$ must be incident to $x,y,z$. For $i\in\{1,2\}$, let $w_i$ be the fourth vertex on the boundary of $f_i$. If $w_1\neq w_2$, then $x$ and $z$ have three distinct common neighbours, contradicting the fact that $H$ is $K_{2,3}$-free. On the other hand, if $w_1$ and $w_2$ coincide, then $f_1$ and $f_2$ have the same boundary which implies that $H$ is a $4$-cycle, and is again a contradiction.
\end{proof}

The following definitions are useful for stating the next lemma. 

\begin{defn}
A vertex $v$ of a graph $G$ is \emph{frozen} by an $F$-colouring $f$ if $g(v)=f(v)$ for every $F$-colouring $g\in \Hom(G,F)$ which reconfigures to $f$. 
\end{defn}

\begin{defn}
An $F$-colouring $f$ of a graph $G$ is said to be \emph{frozen} if every vertex of $G$ is frozen by $f$. 
\end{defn}

\begin{defn}
Say that a graph $F$ is \emph{stiff} if the identity map on $V(F)$ is a frozen $F$-colouring of $F$. 
\end{defn}

\begin{obs}
\label{nbDom}
A graph $F$ is not stiff if and only if there exists distinct vertices $u,v\in V(F)$ such that $N(u)\subseteq N(v)$. 
\end{obs}

\begin{lem}
\label{idFrozen}
$H$ is stiff. 
\end{lem}

\begin{proof}
Suppose not. Then, by Observation~\ref{nbDom}, there exists $u,v\in V(H)$ such that $N(u)\subseteq N(v)$. By Lemma~\ref{minDegis3}, $u$ has degree at least three. Thus, $u$ and $v$ have at least three distinct common neighbours, which contradicts the fact that $H$ is $K_{2,3}$-free. 
\end{proof}

\subsection{Not-Both-One Gadget for Quadrangulations}

The basic idea underlying the proof of Lemma~\ref{nboLem} is that if we take a homomorphism of a path $P=y_1\cdots y_m$ to the ``ladder'' graph $L$ as in Figure~\ref{fig:nbo} such that the vertex $y_i$ is constrained to map to $c_i$ or $d_i$, then the pair  $(y_1, y_m)$ can map only to one of $(d_1, d_m)$,  $(d_1, c_m)$ or $(c_1,c_m)$, and one can easily reconfigure between such patterns. As we will see, constraining $y_i$ to map to $c_i$ or $d_i$ is easy using frozen colourings and $K_{2,3}$-freeness. The harder part is finding an image of the graph $L$ in $H$ with vertices $(c_1,d_1,c_m,d_m) = (a_1,a_0,b_0,b_1)$. This takes most of the subsection, and is done with the directed graph $\Phi$ defined in Definition~\ref{PhiQuad}. Note that a path in the directed graph $\Phi$ corresponds to an image of the graph $L$ in $H$, stronger than a homographic image: we also insist pairs in $L$ connected by dashed edges map to vertices that are across from one another on a face of $H$.

\begin{figure}[htbp]
\begin{center}
\begin{tikzpicture}
   \newdimen\R
   \R=1.65cm
   \newdimen\smallR
   \smallR=1.0cm
   
   \draw\foreach \x in {1,2,3} {
     (1+2*\x,3) node [smallblack, label={[label distance=0pt]90:{$y_{\x}$}}] (u\x){}
   };
   \draw (7.5,3) node (u35){};
   \draw (8.5,3) node (u45){};
   \draw (9,3) node [smallblack, label={[label distance=0pt]90:{$y_{m}$}}] (u5){};
   \draw(u1)--(u2)--(u3)--(u35);
   \draw(8,3) node (){$\cdots$};
   \draw(u45)--(u5);
   \draw\foreach \x in {1,2,3,4,5} {
     (2*\x,2) node [smallblack] (d\x){}
     (2*\x,0) node [smallblack] (c\x){}
   };
   \draw (8.5,2) node (d45){};
   \draw(9,2) node (){$\cdots$};
   \draw (9.5,2) node (d55){};
   \draw (6.5,0) node (c35){};
   \draw(7,0) node (){$\cdots$};
   \draw (7.5,0) node (c45){};
   \draw(8,1) node (){$\cdots$};
   \draw(d1)--(d2)--(d3)--(d4)--(d45);\draw(d55)--(d5);
   \draw(c1)--(c2)--(c3)--(c35);\draw(c45)--(c4)--(c5);
   \foreach [count=\n from 1] \x in {2,3,4,5} {
     \draw[dashed](c\n)--(d\x); };
   \draw\foreach \x in {1,2,3,5}{(c\x)--(d\x)};
   \draw (8,.5) node (cc){};
   \draw (8,1.5) node (dd){};
   \draw (c4)--(cc);
   \draw (dd)--(d4);
   \draw\foreach [count=\n from 1] \x in {1,2,3,m} {
     (2*\n+2.3,1.7) node (){$d_\x$}
     (2*\n-.3,.3) node (){$c_\x$}
     };
     \draw(1,1) node (){$L$};
     \draw(2,4) node (){$P$};
\end{tikzpicture}
\end{center}

\caption{Graphs $P$ and $L$ motivating $\Phi$ of Defintion \ref{PhiQuad}.}
\label{fig:nbo}
\end{figure}
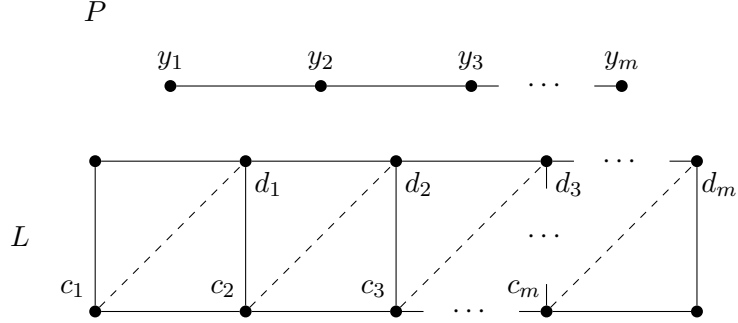
   
We start with the following definitions.

\begin{defn}
Let $A\subseteq V(H)^2$ be the set of all ordered pairs $(a,b)$ such that $a$ is across from $b$.
\end{defn}

Note that because $H$ is $K_{2,3}$-free, there is a unique face incident to both $a$ and $b$ for
any pair $(a,b) \in A$.

\begin{defn}
\label{PhiQuad}
Let $\Phi$ be the directed graph with vertex set $A$ where there is an arc from $(a,b)$ to $(c,d)$ if $ac,bc,bd\in E(H)$ and $ad\notin E(H)$. 
\end{defn}

Figure \ref{fig:nbo} shows a `path' in $\Phi$ from $(c_1,d_1)$ to $(c_m,d_m)$.  Observe that 
the following is immediate from the definition of $\Phi$. 

\begin{obs}
\label{reverse}
There is an arc from $(a,b)$ to $(c,d)$ in $\Phi$ if and only if there is an arc from $(d,c)$ to $(b,a)$ in $\Phi$.
\end{obs}

The following lemma highlights the utility of $\Phi$ in proving Lemma~\ref{nboLem}.

\begin{lem}
\label{pathImpliesGadget}
A directed path from $(a_1,a_0)$ to $(b_0,b_1)$ in $\Phi$ yields a $\{(a_0,b_0),(a_1,b_0),(a_0,b_1)\}$-gadget.
\end{lem}

\begin{proof}
Let $(c_1,d_1),\dots,(c_m,d_m)$ be a directed path in $\Phi$ with $(c_1,d_1)=(a_1,a_0)$ and $(c_m,d_m)=(b_0,b_1)$. We initiate the construction of the gadget $Y(x_1,x_2)$ with  a copy $H^*$ of $H$ where the vertex of $H^*$ corresponding to a vertex $v\in V(H)$ is denoted by $v^*$. We then add a path $y_1\cdots y_m$ and edges from $y_i$ to the copies of the two common neighbours of $c^*_i$ and $d^*_i$ in $H^*$ for $1\leq i\leq m$. Define the signal vertex $x_1$ to be $y_1$ and the signal vertex $x_2$ to be $y_m$. Define the mappings $\zeta_{(a_0,b_0)}, \zeta_{(a_1,b_0)}$ and $\zeta_{(a_0,b_1)}$ so that the copy of $H^*$ is coloured by the identity map $v^*\mapsto v$ and the path $y_1\cdots y_m$ is coloured by 
\[\zeta_{(a_0,b_0)}(y_i):=\begin{cases}d_i & \text{if }i=1,\\
c_i & \text{otherwise}.\end{cases}\]
\[\zeta_{(a_1,b_0)}(y_i) := c_i\text{ for }1\leq i\leq m,\]
\[\zeta_{(a_0,b_1)}(y_i):=d_i\text{ for }1\leq i\leq m.\]
Note that these mappings are indeed $H$-colourings of $Y$ by definition of $\Phi$ and that they satisfy condition (\ref{patterns}) of Definition~\ref{patternDef}. 

We observe that, by Lemma~\ref{idFrozen}, every vertex of $H^*$ is frozen by each of the homomorphisms $\zeta_p$. Thus, if $\psi$ reconfigures to $\zeta_p$ for some $p$, then $\psi(v^*)=v$ for each $v\in V(H)$. Since $H$ is $K_{2,3}$-free, the two vertices of $H$ adjacent to both $c_i$ and $d_i$ do not have a third common neighbour, and so we get that $\psi(y_i)\in \{c_i,d_i\}$. In particular, $\psi(y_1)\in\{a_0,a_1\}$ and $\psi(y_m)\in \{b_0,b_1\}$. So, to verify condition (\ref{stuck}) of Definition~\ref{patternDef}, we need only to show that it cannot be the case that $\psi(y_1)=a_1$ and $\psi(y_m)=b_1$. If we have $\psi(y_1)=a_1 = c_1$, then we must have $\psi(y_2)=c_2$ because $c_1$ is not adjacent to $d_2$ (by construction of $\Phi$). Repeating the same argument, we get $\psi(y_3)=c_3$, $\psi(y_4)=c_4$, and so on. In particular, we must have $\psi(y_m)=c_m = b_0$, as desired.

Finally, we check condition (\ref{patternTransition}) of Definition~\ref{patternDef}. By symmetry of reconfiguration sequences, it suffices to consider $p=(a_0,b_0)$ and $q=(a_1,b_0)$ or $(a_0,b_1)$. The former case is trivial as $\zeta_{(a_0,b_0)}$ and $\zeta_{(a_1,b_0)}$ differ only on $y_1$. In the latter case, we start with $\zeta_{(a_0,b_0)}$ and change the colours of each $y_i$ for $2\leq i\leq m$, one by one, from $c_i$ to $d_i$. Each of the intermediate mappings is an $H$-colouring of $Y$ which differs from the previous one on a unique vertex. After all of these changes have been made, we arrive at $\zeta_{(a_0,b_1)}$. Also, every $H$-colouring in this sequence maps $x_1=y_1$ to $a_0$ and $x_2=y_m$ to either $b_0$ or $b_1$, as required. This completes the proof. 
\end{proof}

Therefore, our goal in proving Lemma~\ref{nboLem} will be to show that there is a directed path in $\Phi$ between any two elements of $A$. To this end, we build up further useful properties of $\Phi$.

\begin{lem}
\label{inOut}
Every vertex of $\Phi$ has in-degree and out-degree equal to $2$. 
\end{lem}

\begin{proof}
Let $(a,b)$ be a vertex of $\Phi$ and let $f$ be the face of $H$ incident to $a$ and $b$. If there is an arc from $(a,b)$ to $(c,d)$, then $c$ must be adjacent to both $a$ and $b$. Thus, since $H$ is $K_{2,3}$-free, there are precisely two choices for $c$; namely, the two other vertices incident to $f$. Now, given a choice of $c$, we may let $f'$ be the unique face with $f'\neq f$ such that the $f'$ is incident to $bc$. Letting $d$ be the non-neighbour of $c$ on the boundary of $f'$, we have that $a$ is not adjacent to $d$ by Lemma~\ref{nonAdj} and so there is an arc from $(a,b)$ to $(c,d)$ in $\Phi$. 

The argument above shows that the out-degree of $(a,b)$ is at least two. To prove that it is exactly two, we need to show that the choice of $d$ is unique once $c$ has been chosen. If not, let $d'\neq d$ be a vertex which is across from $c$, adjacent to $b$ and not adjacent to $a$. Let $f''$ be the face whose boundary contains $c$ and $d'$ and note that $f''\notin\{f,f'\}$ as $d'$ is not adjacent to $a$ and is not equal to $d$. Let $x,y$ be the other two vertices on the boundary of $f''$. Then $c$ and $d'$ are adjacent to all three of $b,x$ and $y$ and so, since $H$ is $K_{2,3}$-free, these three vertices cannot be distinct. Since $x\neq y$, we get that, without loss of generality, $b=x$. However, we now have that the edge $bc$ is on the boundary of three distinct faces, namely $f,f'$ and $f''$, contradicting the fact that $H$ is a planar quadrangulation and completing the proof.

Finally, to see that the in-degree of each vertex is also equal to two, we simply use the fact that the out-degree of every vertex is two and apply Observation~\ref{reverse}. 
\end{proof}

The following definition will be helpful in further analysing $\Phi$. 

\begin{defn}
  Let $\Gamma$ be the graph with vertex set $A$ where $(a,b)$ is adjacent to $(c,d)$ if
\begin{enumerate}[(a)]
\item the face $f$ incident to $a$ and $b$ is distinct from the face $f'$ incident to $c$ and $d$, 
\item there is exactly one edge incident to both $f$ and $f'$, 
\item $a$ is adjacent to $c$, and 
\item $b$ is adjacent to $d$. 
\end{enumerate}
\end{defn}

From this definition, if is clear that for  $(a,b)$ in $A$, if the face $f$ incident to $a$ and $b$ shares
  an edge with a face $f'$, then there is a pair $(c,d)$ $A$ with $f'$ incident to $c$ and $d$,
  such that $(a,b)$ is adjacent to $(c,d)$ in $\Gamma$.

As it turns out, the undirected graph underlying $\Phi$ is precisely $\Gamma$. 

\begin{lem}
\label{underlying}
Let $(a,b),(c,d)\in A$. Then $(a,b)$ is adjacent to $(c,d)$ in $\Gamma$ if and only if there is an arc from $(a,b)$ to $(c,d)$ or an arc from $(c,d)$ to $(a,b)$ in $\Phi$. 
\end{lem}

\begin{proof}
First, suppose that $(a,b)$ is adjacent to $(c,d)$ in $\Gamma$. Let $f$ and $f'$ be the faces of $H$ incident to $a,b$ and $c,d$, respectively, and let $e$ be the common edge of the boundaries of $f$ and $f'$. Note that the fact that $a$ is adjacent to $c$ and $b$ is adjacent to $d$ implies that $\{a,b\}\cap \{c,d\}=\emptyset$. Thus, we must have that exactly one of $c$ or $d$ is incident with $e$ and, likewise, exactly one of $a$ or $b$ is incident with $e$. By Lemma~\ref{nonAdj}, the vertex of $\{a,b\}$ that is not incident with $e$ cannot be adjacent to the vertex of $\{c,d\}$ that is not incident with $e$. Thus, by definition of $\Gamma$, this pair cannot be $a$ and $c$, nor can it be $b$ and $d$. On the other hand, it is easily checked (using Lemma~\ref{nonAdj}) that if $e=bc$, then there is an arc from $(a,b)$ to $(c,d)$ in $\Phi$ and, if $e=ad$, then there is an arc from $(c,d)$ to $(a,b)$ in $\Phi$. 

Now, for the other direction, suppose, without loss of generality, that there is an arc from $(a,b)$ to $(c,d)$ in $\Phi$. Let $f$ be the face incident to $a,b$ and $f'$ be the  face incident to $c,d$. Since $a$ is adjacent to $c$ but neither $c$ nor $a$ is adjacent to $d$, we must have $f\neq f'$. Also, since $H$ has no $K_{2,3}$, it must be the case that $c$ is incident to $f$ and $b$ is incident to $f'$. Thus, the boundaries of $f$ and $f'$ share the edge $bc$ and, by Lemma~\ref{nonAdj}, this is the unique such edge. So, we can conclude that $(a,b)$ is adjacent to $(c,d)$ in $\Gamma$. 
\end{proof}

Next, we prove that $\Gamma$ is connected. Note that there are some subtleties here to be aware of. In particular, the fact that every finite quadrangulation has a vertex of degree at most three will be crucial. For infinite quadrangulations of the plane, the graph $\Gamma$ is not connected in general; e.g. if $H$ were the infinite square grid, then $\Gamma$ would contain exactly four connected components, each of which is itself isomorphic to an infinite square grid. Also, the graph $\Gamma$ can sometimes be disconnected if $H$ is a finite quadrangulation of the torus; consider, for example, the Cartesian product of two cycles. Thus, in some sense, the structure of the sphere (in particular, the fact that it has positive curvature) is important for our proof to go through.

\begin{lem}
\label{GammaConn}
$\Gamma$ is connected. 
\end{lem}

\begin{proof}
Let vertices $0,1,2,3,\alpha_{1,2},\alpha_{2,3}$ and $\alpha_{3,1}$ and faces $f_1,f_2$ and $f_3$ of $H$ be defined as in the previous section.  We show that every $(a,b)\in A$ admits a path to $(0,1)$ in $\Gamma$, which will complete the proof.

Let $(a,b)\in A$ be arbitrary and let $f'$ be the face whose boundary contains $a$ and $b$. We denote the four faces of $H$ whose boundaries share a unique edge with the boundary of $f'$ by $f_1',\dots,f_4'$. In light of Lemma~\ref{underlying}, Lemma~\ref{inOut}  gives us that $(a,b)$ has four neighbours $(c_1,d_1),\dots,(c_4,d_4)$ in $\Gamma$, where $c_i$ and $d_i$ are on the boundary of $f_i'$. Thus, by the connectedness of the planar dual of $H$, we get that there is a path in $\Gamma$ starting at $(a,b)$ and terminating at a pair $(s,t)$ on the boundary of the face $f_1$. We show that there is a path from every such $(s,t)$ to $(0,1)$. 

Of course, if $(s,t)=(0,1)$, then we are simply done. If $(s,t)=(\alpha_{3,1},\alpha_{1,2})$, then it is adjacent to $(0,2)$, which is adjacent to $(\alpha_{2,3},\alpha_{3,1})$, which is adjacent to $(0,1)$ and we are done. The proof in the case $(s,t)=(\alpha_{1,2},\alpha_{3,1})$ is similar. If $(s,t)=(1,0)$, then we see that $(1,0)$ is adjacent to $(\alpha_{3,1},\alpha_{2,3})$, which is adjacent to $(0,2)$, which is adjacent to $(\alpha_{3,1},\alpha_{1,2})$ which was already shown to admit a path to $(0,1)$. This completes the proof.
\end{proof}

We are now ready to prove Lemma~\ref{nboLem}.

\begin{proof}[Proof of Lemma~\ref{nboLem}]
By Lemmas~\ref{underlying} and~\ref{GammaConn}, the undirected graph underlying $\Phi$ is connected and, by Lemma~\ref{inOut}, every vertex of $\Phi$ has in-degree equal to its out-degree. Thus, $\Phi$ has an Eulerian circuit. In particular, this implies that there is a directed path between any two vertices of $\Phi$, and so we are done by Lemma~\ref{pathImpliesGadget}. 
\end{proof}

\subsection{Not-All-Zero Gadget for Quadrangulations}

We will now use Lemma~\ref{nboLem} to prove Lemma~\ref{nazLem}. We remark that, in a subtle way, this construction relies strongly on the fact that $0$ was chosen to be a vertex of degree \emph{exactly} three in $H$. In topological language, the proof relies strongly on the fact that the ``link'' of the vertex $0$ is a triangle.

\begin{proof}[Proof of Lemma~\ref{nazLem}]
We construct a not-all-zero gadget $Z(z_1,z_2,z_3,z_4)$. First, applying Lemma~\ref{nboLem}, we construct two $\{(1,2),(0,2),(1,0)\}$-gadgets $W_1(z_1,w_1)$ and $W_3(z_3,w_3)$, 
and two $\{(1,3),(0,3),(1,0)\}$-gadgets $W_2(z_2,w_2)$ and $W_4(z_4,w_4)$ disjointly. Next, add a disjoint copy $H^*$ of $H$ in which the vertex of $H^*$ corresponding to a vertex $v\in V(H)$ is denoted $v^*$ and add four new vertices $x_{1,2},y_{1,2},x_{3,4}$ and $y_{3,4}$ such that
\begin{itemize}
\item $x_{1,2}$ is adjacent to $w_1,w_2,0^*$ and $y_{1,2}$ and
\item $y_{1,2}$ is adjacent to $\alpha_{1,2}^*,\alpha_{3,1}^*$ and $x_{1,2}$.
\end{itemize}
Finally, we add a $\{(1,0),(0,1),(1,1)\}$-gadget $Y(y_{1,2},y_{3,4})$. See Figure~\ref{nazFig}.

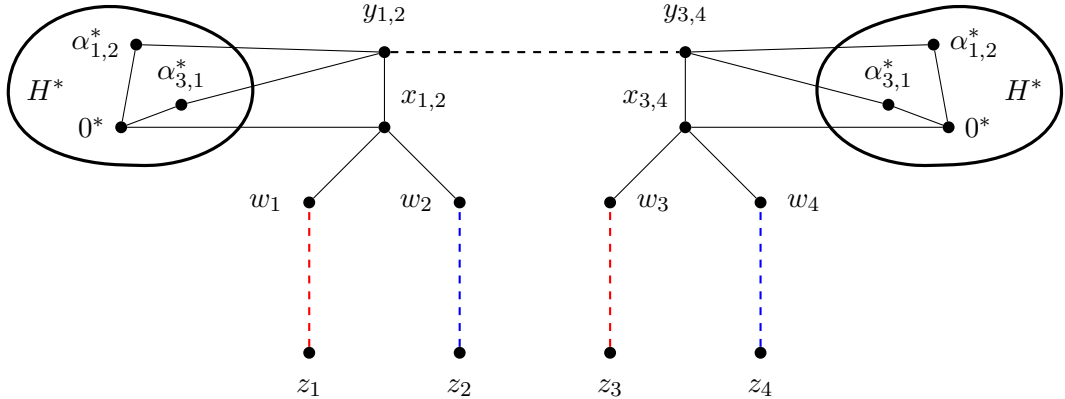
\begin{figure}[htbp]
\begin{center}
\begin{tikzpicture}

\draw (0,0) node [smallblack,label={[label distance=4pt]270:{$z_1$}}] (x1) {};
\draw (2,0) node [smallblack,label={[label distance=4pt]270:{$z_2$}}] (x2) {};
\draw (4,0) node [smallblack,label={[label distance=4pt]270:{$z_3$}}] (x3) {};
\draw (6,0) node [smallblack,label={[label distance=4pt]270:{$z_4$}}] (x4) {};

\draw (0,2) node [smallblack,label={[label distance=4pt]180:{$w_1$}}] (w1) {};
\draw (2,2) node [smallblack,label={[label distance=4pt]180:{$w_2$}}] (w2) {};
\draw (4,2) node [smallblack,label={[label distance=4pt]0:{$w_3$}}] (w3) {};
\draw (6,2) node [smallblack,label={[label distance=4pt]0:{$w_4$}}] (w4) {};
\draw (1,3) node [smallblack,label={[label distance=1pt]45:{$x_{1,2}$}}] (y12) {};
\draw (5,3) node [smallblack,label={[label distance=1pt]135:{$x_{3,4}$}}] (y34) {};
\draw (1,4) node [smallblack,label={[label distance=4pt]90:{$y_{1,2}$}}] (z12) {};
\draw (5,4) node [smallblack,label={[label distance=4pt]90:{$y_{3,4}$}}] (z34) {};

\draw(w1)--(y12);
\draw(w2)--(y12);
\draw(y12)--(z12);

\draw(w3)--(y34);
\draw(w4)--(y34);
\draw(y34)--(z34);
\draw[thick, dashed](z12)--(z34);
\draw[red,thick,dashed](x1)--(w1);
\draw[red,thick,dashed](x3)--(w3);
\draw[blue,thick,dashed](x2)--(w2);
\draw[blue,thick,dashed](x4)--(w4);

\draw[very thick] (-0.75,3.5) to [closed, quick curve through={(-0.75,3.5) . . (-2,4.5) . . (-2.5,4.6) . .  (-4,3.5) . . (-2.5,2.5) . . (-2,2.5) . . (-0.75,3.5)}] (-0.75,3.5);

\draw (-3.5,3.5) node (H1) {$H^*$};

\draw (-2.5,3) node[smallblack,label={[label distance=0pt]180:{$0^*$}}] (01) {};
\draw (-2.3,4.1) node[smallblack,label={[label distance=0pt]180:{$\alpha_{1,2}^*$}}] (alpha121) {};
\draw (-1.7,3.3) node[smallblack,label={[label distance=-1pt]90:{$\alpha_{3,1}^*$}}] (alpha311) {};

\draw(01)--(alpha121);
\draw(01)--(alpha311);
\draw(z12)--(alpha121);
\draw(z12)--(alpha311);
\draw(y12)--(01);

\draw[very thick] (6.75,3.5) to [closed, quick curve through={(6.75,3.5) . . (8,4.5) . . (8.5,4.6) . .  (10,3.5) . . (8.5,2.5) . . (8,2.5) . . (6.75,3.5)}] (6.75,3.5);

\draw (9.5,3.5) node (H1) {$H^*$};

\draw (8.5,3) node[smallblack,label={[label distance=0pt]0:{$0^*$}}] (02) {};
\draw (8.3,4.1) node[smallblack,label={[label distance=0pt]0:{$\alpha_{1,2}^*$}}] (alpha122) {};
\draw (7.7,3.3) node[smallblack,label={[label distance=-1pt]90:{$\alpha_{3,1}^*$}}] (alpha312) {};

\draw(02)--(alpha122);
\draw(02)--(alpha312);
\draw(z34)--(alpha122);
\draw(z34)--(alpha312);
\draw(y34)--(02);

\end{tikzpicture}
\end{center}
\caption{An illustration of the not-all-zero gadget. Thin solid black lines represent edges, thick solid black closed curves represent the copy $H^*$ of $H$ (drawn twice for clarity), red dashed lines connect signal vertices of $\{(1,2),(0,2),(1,0)\}$-gadgets, blue dashed lines connect signal vertices of $\{(1,3),(0,3),(1,0)\}$-gadgets and the black dashed line connects signal vertices of a  $\{(1,0),(0,1),(1,1)\}$-gadget. }
\label{nazFig}
\end{figure}

Now, for $p=(p_1,\dots,p_4)\in\{0,1\}^4\setminus\{(0,0,0,0)\}$, we define $\zeta_p$ as follows. The copy $H^*$ of $H$ is coloured according to the identity colouring $v^*\mapsto v$. If $p_1=0$, then we colour $W_1(z_1,w_1)$ according to $\zeta_{(0,2)}$ and similar for $W_3(z_3,w_3)$ if $p_3 = 0$. Similarly, if $p_2=0$, then we colour $W_2(z_2,w_2)$ according to $\zeta_{(0,3)}$ and similar for $W_4(z_4,w_4)$ if $p_4 = 0$. For each $i$ such that $p_i=1$, we colour the gadget $W_i(z_i,w_i)$ with $\zeta_{(1,0)}$. We colour $x_{1,2}$ with $\alpha_{1,2}$ if $p_2=1$, with $\alpha_{3,1}$ if $p_2=0$ and $p_1=1$ and with $\alpha_{2,3}$ otherwise. Also, colour $y_{1,2}$ with $1$ if at least one of $p_1$ or $p_2$ is equal to $1$ and $0$ otherwise. The colouring of $x_{3,4}$ and $y_{3,4}$ is similar. Since at least one of $p_1,\dots,p_4$ is equal to $1$, we know that one of $y_{1,2}$ or $y_{3,4}$ is mapped to $1$. We colour the gadget $Y(y_{1,2},y_{3,4})$ with one of the colourings $\zeta_{(1,0)},\zeta_{(0,1)}$ or $\zeta_{(1,1)}$ depending on which of $y_{1,2}$ or $y_{3,4}$ has already been coloured with $1$. One can easily check that this definition of $\zeta_p$ is a valid $H$-colouring and that $\zeta_p(z_i)=p_i$ for $i=1,2,3,4$; thus condition (\ref{patterns}) of Definition~\ref{patternDef} is satisfied. The rest of the proof consists of verifying that the other two conditions of Definition~\ref{patternDef} hold.

Since each of the vertices $z_1,z_2,z_3,z_4$ is contained in either a $\{(1,2),(0,2),(1,0)\}$-gadget or a $\{(1,3),(0,3),(1,0)\}$-gadget which is mapped to a canonical colouring, we know that $\psi(z_i)\in \{0,1\}$ for all $\psi$ which reconfigures to any $\zeta_p$ for $p\in\{0,1\}^4\setminus\{(0,0,0,0)\}$. Thus, to verify condition (\ref{stuck}) of Definition~\ref{patternDef} we need only show that no $H$-colouring $\psi$ which reconfigures to some $\zeta_p$ can map all of $z_1,z_2,z_3,z_4$ to zero. By definition of the $\{(1,2),(0,2),(1,0)\}$-gadget and $\{(1,3),(0,3),(1,0)\}$-gadget, if $\psi(z_1)=\psi(z_2)=0$, then $\psi(w_1)=2$ and $\psi(w_2)=3$. Also, by Lemma~\ref{idFrozen}, $\psi$ must colour $H^*$ according to the identity map. Thus, $x_{1,2}$ must map to a common neighbour of $0,2,3$, which implies that $\psi(x_{1,2})=\alpha_{2,3}$ as $H$ is $K_{2,3}$-free. From this, we get that $y_{1,2}$ must map to $0$, as this is the only common neighbour of $\alpha_{1,2}, \alpha_{3,1}$ and $\alpha_{2,3}$, again by $K_{2,3}$-freeness. Applying the same argument starting with $z_3$ and $z_4$ gives us that $y_{3,4}$ maps to $0$ as well, but this contradicts the definition of the $\{(1,0),(0,1),(1,1)\}$-gadget. Therefore, condition (\ref{stuck}) of Definition~\ref{patternDef} holds. 

Finally, we move on to condition (\ref{patternTransition}) of Definition~\ref{patternDef}. Let $p,q\in\{0,1\}^4\setminus\{(0,0,0,0)\}$ differ on a unique coordinate. By symmetry and without loss of generality we can assume that $p_1=0$ and $q_1=1$ (note that this is indeed without loss of generality since we can always swap the names of vertices $2$ and $3$ of $H$). We begin by reconfiguring the colouring of the gadget between $z_1$ and $w_1$ from $\zeta_{(0,2)}$ to $\zeta_{(1,2)}$ without changing the colour of $w_1$. Then, without changing the colour of $z_1$, we reconfigure it to $\zeta_{(1,0)}$. At this point, the current colouring differs from $\zeta_q$ only on $x_{1,2}$ and vertices of the gadget $Y(y_{1,2},y_{3,4})$. However, we may simply change the colour of  $x_{1,2}$ to  $\zeta_q\left(x_{1,2}\right)$ at this point since, by definition of  $\zeta_q$, it is either $\alpha_{1,2}$ or $\alpha_{3,1}$, both of which are compatible with the colour of $y_{1,2}$. As a final step, we reconfigure the colouring of $Y(y_{1,2},y_{3,4})$ to match $\zeta_q$. This completes the proof.
\end{proof}

\section{Quadrangulations of Projective Planes}
\label{wheelSection}

  Our first goal in this section is to deduce Theorem~\ref{ppCor} from  Theorem~\ref{quadThm}.    We then 
  adapt the argument to even wheels to finish off Theorem \ref{wheelThm}.  We start with some background about
  embeddings into projective plane, and recall a couple of useful ideas from \cite{WrochnaMasters}.


 Recall that the projective plane $\P^2$ can be expressed as the quotient $q: S^2 \to \P^2$ of the sphere $S^2$ modulo the equivalence relation identifying antipodal pairs of points.
 Though we use this point of view in our proof, for
 intuition and figures, we  use  the equivalent representation of $\P^2$  as the disk $B^2$ modulo the equivalence identifying antipodal pairs of points on the boundary. 

  One of the main features of  interest of the non-orientable surface $\P^2$ is that one can embed the cycle $C_r$ in it, for $r$ odd,  so that there  is exactly one face of length $2r$.  
 Indeed, recall that to find the length of a face in an embedding of a graph, we walk around it, keeping our right hand on the boundary, and count the number of vertices we encounter, with repetition, until we get back to where we started.  Referring to the first picture of Figure \ref{fig:embedpp1} in which we have embedded $C_5$ in the  projective plane, we see that there is one face of length $10$ as follows. Start just below vertex $1$ on the northern (top) hemisphere of the figure, and keeping your right hand on the edge, traverse the first edge to arrive at vertex $2$.  From there, the edge on your right is the edge from $2$ to $3$e drawn in the southern hemisphere. Traversing that edge, you then continue from $3$ to $4$ on the west.  You do not get back to you starting position  until you have traversed ten edges.    Playing the same game with the second picture, you see that you arrive back at vertex $1$ after five steps, but you are on the other side of it. You must traverse ten edges to get back to where you started.

\begin{figure}[htbp]
\begin{center}
\begin{tikzpicture}[every node/.style = {smallblack}]
   \newdimen\R
   \R=1.85cm
   \draw [thick,dotted] (0,0) circle [radius=\R];
   \foreach \i[evaluate=\i as \qi using {90 + (\i - 1)* 72*3}, 
                    evaluate=\i as \ri using  {270 + (\i -1) *72*3}] in {1,...,5}{
      \draw(\qi:\R) node[label =  \qi:$\i$] (u\i) {};
      \draw(\ri:\R) node[label = \ri: $\i$] (v\i) {};}
   \foreach \i/\j in {1/2,2/3,3/4,4/5,5/1}{\draw (u\i) edge (v\j);}

  \begin{scope}[xshift = 8cm]
   \draw [thick,dotted] (0,0) circle [radius=\R];
   \foreach \i[evaluate=\i as \qi using {90 + (\i - 1)* 72*3}, 
                    evaluate=\i as \ai using  {108 + (\i - 1)*72*3},
                    evaluate=\i as \bi using  {288 + (\i - 1)*72*3} ] in {1,...,5}{
      \draw(\qi:\R-.3cm) node[label =  \qi:$\i$] (u\i) {};
      \draw(\ai:\R) node[draw = none, fill = none] (a\i) {};
      \draw(\bi:\R) node[empty] (b\i){};}
      \foreach \i/\j in {1/2,2/3,3/4,4/5,5/1}{\draw (u\i) edge (a\i);
                                                               \draw (b\i) edge  (u\j);}
  \end{scope}
  \end{tikzpicture}
  \end{center}
\caption{Two drawings of the same embedding of $C_5$ in the projective plane}\label{fig:embedpp1}
  \end{figure}
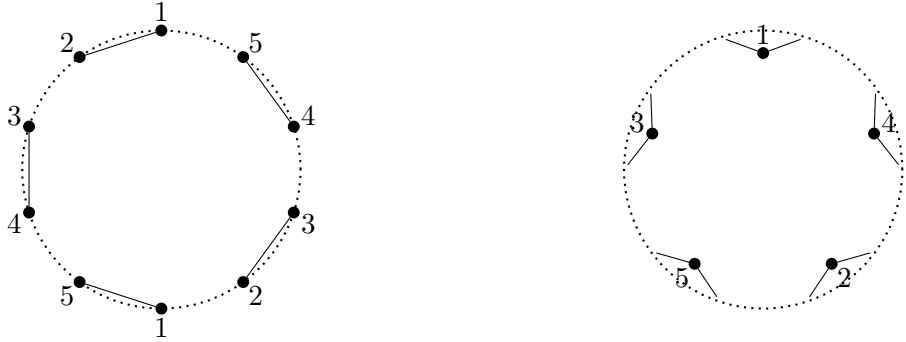  
  
Recall that an embedding of a graph $H$ in a surface is a quadrangulation if every face is of length four. An odd wheel can be drawn as a quadrangulation of $\P^2$ by embedding its outer cycle in the manner described above; see Figure \ref{fig:embedpp2}.

 \begin{figure}
\begin{center}
\begin{tikzpicture}[every node/.style = {smallblack}]
   \newdimen\R
   \R=1.85cm
   \draw [thick,dotted] (0,0) circle [radius=\R];
   \draw (0,0) node[label ={[label distance = .3cm] -5: $\alpha$}] (alpha) {};
   \foreach \i[evaluate=\i as \qi using {90 + (\i - 1)* 72*3}, 
                    evaluate=\i as \ri using  {270 + (\i -1) *72*3}] in {1,...,5}{
      \draw(\qi:\R) node[label =  \qi:$\i$] (u\i) {};
      \draw(\ri:\R) node[label = \ri: $\i$] (v\i) {};}
   \foreach \i/\j in {1/2,2/3,3/4,4/5,5/1}{\draw (u\i) edge (v\j);
                                                             \draw (alpha) edge (u\i);}

  \begin{scope}[xshift = 8cm]
    \draw [thick,dotted] (0,0) circle [radius=\R];
   \draw (0,0) node (alpha) {};
   \foreach \i[evaluate=\i as \qi using {90 + (\i - 1)* 72}, 
                    evaluate=\i as \ri using  {126 + (\i -1) *72}] in {1,...,5}{
      \draw(\qi:\R) node (u\i) {};
      \draw(\ri:\R/2) node (m\i) {};
      \draw(\ri:\R) node (v\i) {};}
   \foreach \i/\j in {1/2,2/3,3/4,4/5,5/1}{\draw (u\i) edge (v\i);
                                                             \draw (alpha) edge (m\i);
                                                             \draw (u\i) edge (m\i);
                                                             \draw (m\i) edge (u\j);}

  \end{scope}
  \end{tikzpicture}
\end{center}
\caption{Embedding of odd wheel and $4$-chromatic Myceilski graph in projective plane}\label{fig:embedpp2}
\end{figure}
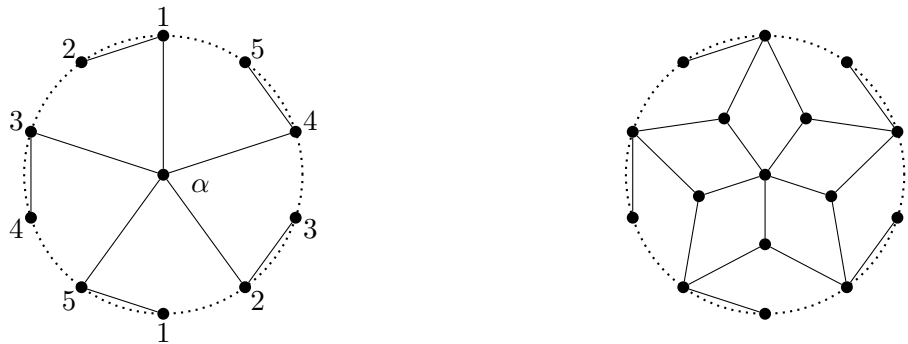

\begin{defn}
For graphs $F_1$ and $F_2$, the \emph{categorical} product (sometimes called the \emph{direct} or \emph{tensor} product) of $F_1$ and $F_2$, denoted $F_1\times F_2$, is the graph on vertex set $V(F_1)\times V(F_2)$ where $(u_1,u_2)$ is adjacent to $(v_1,v_2)$ if and only if $u_1v_1\in E(F_1)$ and $u_2v_2\in E(F_2)$. 
\end{defn}

The product $H \times K_2$ of $H$ with an edge is  often referred to as the  ``bipartite double cover'' of $H$, among other names.  It is well known and easy to see that if $H$ is bipartite, then  $H \times K_2$ is two disjoint copies of $H$.  It is well known, and easy to show, that for a graph $G$ and homomorphisms $\phi_1: G \to F_1$ and $\phi_2: G \to F_2$, the map $\phi_1 \times \phi_2 : G \to F_1 \times F_2$ defined by $v \mapsto (\phi_1(v), \phi_2(v))$ is a homomorphism.
(In fact it is known to be the unique homomorphism that commutes with the projections of the product onto its factors.)  

\begin{lem}[Wrochna~\cite{WrochnaMasters}]
\label{xK2}
For every graph $H$, the \Hrec{(H\times K_2)} problem is polynomially equivalent to the restriction of the \Hrec{H} problem to bipartite instances.
\end{lem}

\begin{proof}
Let $V(K_2)=\{1,2\}$, let $\pi_1: V(H)\times V(K_2)\to V(H)$ be the projection onto the first coordinate and let $\pi_2:V(H)\times V(K_2)\to V(K_2)$ be the projection onto the second coordinate. Note that $\pi_1$ and $\pi_2$ are homomorphisms from $H\times K_2$ to $H$ and $K_2$, respectively. Thus, if $(G,f,g)$ is an instance of \Hrec{(H\times K_2)}, then $G$ admits a $K_2$-colouring (equivalently, $G$ is bipartite) by composing $f$ with $\pi_2$. 

Now, let $G$ be any bipartite graph and let $A,B$ be the two sets of the bipartition. Given a homomorphism $f$ from $G$ to $H\times K_2$, we can compose $f$ with $\pi_1$ to get a homomorphism $f'$ to $H$. On the other hand, if $f'$ is a homomorphism from $G$ to $H$, we can let $f:V(G)\to V(H)\times V(K_2)$ be defined so that $f(v)=(f'(v),1)$ if $v\in A$ and $f(v)=(f'(v),2)$ if $v\in B$. It is not hard to show that both of these transformations preserve the ``reconfigures to'' relation, and so the \Hrec{(H\times K_2)} problem is polynomial-time equivalent to the \Hrec{H} problem restricted to bipartite instances. 
\end{proof}

Theorem~\ref{ppCor} is now immediate from Theorem \ref{quadThm} and 
Lemma \ref{xK2} by the following fact.

\begin{fact}
 Let $H$ be a non-bipartite quadrangulation of the projective plane.  The  graph $H \times K_2$ can be embedded as a quadrangulation of the sphere.   
\end{fact}

\begin{proof}
  Given a  embedding of $H$ in the projective plane, by making small  perterbations we may assume that no vertex of $H$ is
  embedded on the boudary of $\P^2$.  Consider the preimage under $q: S^2 \to \P^2$ of this embedding of $H$. It is
  clearly an embedding on $S^2$ of some graph $H'$; and $q$ clearly induces a graph homomorphism $q: H' \to H$.   For any vertex $v$ 
 of $H$, let $q^{-1}(v) = \{v_N, v_S\}$ where the subscript designates the hemisphere of $S^2$ in which the vertex is embedded.  What remains to be proven is that $H'$ is a quadrangulation of $S^2$ and that it is isomorphic to
 $H \times K_2$. 
  To see that it is a quadrangulation,  observe that a closed (topological) path $P$ in $S^2$ defines a closed path $q(P)$ in $\P^2$.  
  The boundary of a face of $H'$ therefore maps to the boundary of a face in $H$ and, by assumption, this is always 
  a quadrangle.

  The graphs  $H'$ and $H \times K_2$ clearly have the same numbers of vertices and edges, so to show they are isomorphic it is enough to show there is a homomorphism between them that is a bijection on the vertices.  As $H'$ is a quadrangulation of $S^2$, it is bipartite, so has a $K_2$-colouring $\phi$.  We claim that the homomorphism 
      \[ q \times \phi: H' \to H \times K_2\text{ defined by } x \mapsto (q(x), \phi(x)) \] 
is a bijection.  Indeed we have that $[q \times \phi](v_N)$ and $[q \times \phi](v_S)$ are both in 
$\{ (v,0), (v,1) \}$ for each $v$, so it is enough to show that $\phi(v_N) \neq \phi(v_S)$ for all $v$ in $V(H)$.  Assume, towards contradiction, that $\phi(v_N) = \phi(v_S)$. Then
 for any neighbour $u$ of $v$ we have that each of $v_N$ and $v_S$ are adjacent to one of $u_N$ and $u_S$, and so 
 $\phi(u_N) = \phi(u_S)$. As $H$ is connected this therefore holds for all vertices of $H$ and so $\phi$ induces
a $K_2$-colouring of $H$, which is impossible.  Thus $\phi(v_N) \neq \phi(v_S)$ for all $v$ in $V(H)$ as needed. 
\end{proof}

As we noted above, this tells us that \Hrec{H} is PSPACE-complete for $H$ an odd wheel. One can construct many non-bipartite quadrangulations of $\P^2$.  Starting with an odd cycle embedded as we show in Figure \ref{fig:embedpp1} one simply has to quadrangulate it, as in the graph on the right of Figure~\ref{fig:embedpp2}. This is the Gr\"otzsch graph, which is obtained from a $5$-cycle by applying the Mycielski construction. More generally, every $4$-chromatic ``generalized Mycielski graph'' (see~\cite{genMyc} for a definition) can be embedded as a quadrangulation of the projective plane, yielding another interesting special case of Theorem~\ref{ppCor}.

We now provide an extra argument which allows us to extend the argument to even wheels of lenth at least six, thereby allowing us to prove Theorem~\ref{wheelThm}. We need to recall one more basic idea about retractions.
 
\begin{defn}
Given a graph $F$ and an induced subgraph $H$ of $F$, a map $\varphi:V(F)\to V(H)$ is called a \emph{retraction} from $F$ to $H$ if $\varphi$ is a homomorphism from $F$ to $H$ and $\varphi(v)=v$ for every $v\in V(H)$. 
\end{defn}

\begin{defn}
An induced subgraph $H$ of a graph $F$ is called a \emph{retract} of $F$ if there is a retraction from $F$ to $H$.
\end{defn}

\begin{lem}
\label{retractHard}
If $H$ is a retract of $F$, then \Hrec{H} reduces to \Hrec{F}. 
\end{lem}
\begin{proof}
Let $\varphi$ be a retraction from $F$ to $H$ and define $\iota:V(H)\to V(F)$ by $\iota(v)=v$ for all $v\in V(H)$. Clearly, $\iota$ is a homomorphism from $H$ to $F$. Note that $\varphi \circ\iota$ is nothing more than the identity map on $H$.

Now, given an instance $(G,f,g)$ of \Hrec{H}, we consider the instance $(G,\iota\circ f,\iota\circ g)$ of \Hrec{F}. Clearly, if $f$ reconfigures to $g$, then composing each map on the reconfiguration sequence with $\iota$ yields a reconfiguration sequence taking $\iota\circ f$ to $\iota\circ g$. On the other hand, if $\iota\circ f$ reconfigures to $\iota\circ g$, then composing each map on the reconfiguration sequence with $\varphi$ (and possibly taking a subsequence) yields a reconfiguration sequence taking $f$ to $g$. This completes the proof.
\end{proof}

With Theorem~\ref{ppCor}, this and the following fact show that $\Hrec{H}$ is PSPACE-complete for even wheels $H = W_{2n}$ as long as $n \geq 3$, proving Theorem \ref{wheelThm}.

\begin{fact}
  For even $k \geq 6$, the graph $W_{k} \times K_2$ retracts to a non-trivial quadrangulation of $S^2$.
  \end{fact}  
\begin{proof}

The graph $W_k\times K_2$ consists of two cycles
\[C_1:=(x_1,1)(x_2,2),\dots,(x_k,2),\]
\[C_2:=(x_1,2)(x_2,1),\dots,(x_k,1)\]
and vertices $(\alpha,1)$ and $(\alpha,2)$ where the neighbourhood of $(\alpha,1)$ is $\left\{(x_1,2),\dots,(x_k,2)\right\}$ and the neighbourhood of $(\alpha,2)$ is $\left\{(x_1,1),\dots,(x_k,1)\right\}$. Let $H_k$ be the subgraph of $W_k\times K_2$ induced by $V(C_1)\cup\{(\alpha,1),(\alpha,2)\}$.
The graph $H_k$ is easily seen to be a non-trivial quadrangulation of the sphere, 
so we will be done if  we can find a retraction from $W_k\times K_2$ to $H_k$. 
Let $\varphi:V(W_k)\times V(K_2)\to V(H_k)$ be defined so that, for $1\leq i\leq k$ and $j\in\{1,2\}$, we have
\[\varphi(x_i,j):= (x_{i+1},j)\text{ if }i\not\equiv j\bmod 2,\]
\[\varphi(x_i,j):= (x_{i},j)\text{ if }i\equiv j\bmod 2,\]
\[\varphi(\alpha,j):=(\alpha,j)\]
(where indices are viewed modulo $k$). It is easily observed that this is, indeed, a retraction from $W_k\times K_2$ to $H_k$ and so we are done. 
\end{proof}

\section{Gadgets for Reflexive Triangulations}
\label{triSection}

In this section, all graphs are assumed to be reflexive, unless otherwise stated. Our goal is to prove that \Hrec{H} is PSPACE-complete when $H$ is a reflexive graph which contains a substructure which resembles a triangulation near a vertex. The following definition is useful for defining the class of graphs that we consider. 

\begin{defn}
Given a stiff reflexive graph $H$, we say that a set $S\subseteq V(H)$ is \emph{listable} if there exists a reflexive $S$-gadget $X(x_1)$.
\end{defn}

The following, somewhat cumbersome but quite broad, definition describes the class of graphs that we will consider. 

\begin{defn}
\label{localDef}
A stiff reflexive graph $H$ is \emph{locally triangulated} around a vertex $0\in V(H)$ if $H$ contains a reflexive subgraph $F$ such that
\begin{enumerate}[(a)]
\item\label{0F} $0\in V(F)$,
\item \label{noK4} the subgraph of $H$ induced by $V(F)$ is $K_4$-free,
\item\label{nbCycle} the neighbourhood of $0$ in $F$ contains a spanning cycle, say $12\cdots k$,
\item\label{beta} for $1\leq i\leq k$, the neighbours $i$ and $i+1$ of $0$ (where vertex labels are viewed modulo $k$) have a common neighbour $\beta_i$ in $V(F)$ which is distinct from $0$,
\item\label{betaCycle} for $1\leq i\leq k$, the neighbourhood of $\beta_i$ in $F$ contains a spanning cycle,
 \item\label{edgeList} every pair $\{u,v\}$ where $uv\in E(F)$ and $u\neq v$ is listable and
\item\label{triList} the sets $\{0,2,3\}$ and $\{0,3,4\}$ are listable.
\end{enumerate}
\end{defn}

The main focus of this section is on proving variants of Lemmas~\ref{nboLemTri} and~\ref{nazLemTri} for graphs which are locally triangulated around a vertex (Lemmas~\ref{nboLemLocal} and~\ref{nazLemLocal} below). Next, we show that locally triangulated graphs generalize $K_4$-free triangulations distinct from the reflexive triangle and, thus, Lemmas~\ref{nboLemTri} and~\ref{nazLemTri} follow from the results of this section. 

\begin{lem}
\label{localImplies}
Let $H$ be a finite reflexive $K_4$-free triangulation which is not isomorphic to a reflexive triangle. Then $H$ is locally triangulated around every vertex $0\in V(H)$. 
\end{lem}

\begin{proof}
We let $F:=H$. Given this, it is trivial that  conditions \eqref{0F} and \eqref{noK4} of Definition~\ref{localDef} hold. In any triangulation, apart from the triangle, the neighbourhood of any given vertex induces a cycle.  So, condition \eqref{nbCycle} of Definition~\ref{localDef} holds. As in Definition~\ref{localDef}, we denote the neighbours of $0$ by $1,\dots,k$ where $12\cdots k$ is a cycle. 

If $H$ is not stiff, then, by Observation~\ref{nbDom} and since $0$ is an arbitrary vertex and $H$ is reflexive, we have, without loss of generality,  $N(0)\subseteq N(1)$. However, if this were the case, then the vertices $0,1,2,3$ would form a reflexive $K_4$, which is a contradiction. So, $H$ is stiff. 

The edge from $i$ to $i+1$ separates the face $f$ incident to $0,i$ and $i+1$ from another face $f'$. Let $\beta_i$ be the vertex incident to $f'$, distinct from $i$ and $i+1$. Since $H$ is not a reflexive triangle, we must have that  $\beta_i$ is distinct from vertex $0$ and so condition \eqref{beta} of Definition~\ref{localDef} holds. The same argument that was used to prove \eqref{nbCycle} shows that \eqref{betaCycle} holds as well. 

Now, given an edge $u,v\in E(H)$, we let $x,y$ be the two vertices, distinct from $u$ and $v$, which are incident to the two faces whose boundary contains the edge $uv$. Then $u,v$ and $x$ are the only common neighbours of $u, v$ and $x$, and so \[N(u)\cap N(v)\cap N(x)\cap N(y)=\{u,v\}.\]
Given this, and the fact that $H$ is stiff, a $\{u,v\}$-gadget can be constructed by simply taking a copy of $H$ coloured by the identity map in all canonical colourings and adding a reflexive signal vertex $x_1$ adjacent to $u,v,x$ and $y$. So condition \eqref{edgeList} of Definition~\ref{localDef} holds. Also, for any triangle $x,y,z$ in $H$, we have
\[N(x)\cap N(y)\cap N(z)=\{x,y,z\}\]
which, via a similar argument to that which was used for \eqref{edgeList}, gives us condition \eqref{triList}. This completes the proof. 
\end{proof}

\begin{rem}
Of course, the condition that $H$ is locally triangulated around a vertex covers a wide variety of graphs beyond $K_4$-free reflexive triangulations. For example, many reflexive graphs embedded on other surfaces (which need not even be triangulations of that surface, i.e., they can have larger faces) satisfy the criteria of Definition~\ref{localDef}. 
\end{rem}

Thus, by Lemma~\ref{localImplies}, the following theorem, proven in the rest of this section, generalizes Theorem~\ref{triThm}.

\begin{thm}
\label{localThm}
If $H$ is  a finite reflexive graph which is locally triangulated around a vertex $0$, then \Hrec{H} is PSPACE-complete when restricted to instances $(G,f,g)$ such that $G$ is reflexive. 
\end{thm}

The aim in the rest of the section is to establish the following two lemmas which imply Theorem~\ref{localThm} via Lemma~\ref{gadgetImpliesHard}.

\begin{lem}
\label{nboLemLocal}
Let $H$ be a finite reflexive graph which is locally triangulated around a vertex $0$ and let $1,\dots,k$ and $\beta_1,\dots,\beta_k$ be as in Definition~\ref{localDef}. Then, for $1\leq i\leq k$, there exists
\begin{itemize}
\item a $\{(0,0),(0,1),(1,0)\}$-gadget,
\item a $\{(1,i),(1,\beta_i),(0,\beta_i)\}$-gadget and
\item a $\{(1,i+1),(1,\beta_i),(0,\beta_i)\}$-gadget.
\end{itemize}
\end{lem}

\begin{lem}
\label{nazLemLocal}
If $H$ is a finite reflexive graph which is locally triangulated around a vertex $0$, then for every neighbour $1$ of $0$ with $0\neq 1$, there exists a not-all-zero gadget.
\end{lem}

\subsection{Not-Both-One Gadget for Locally Triangulated Graphs}

In the rest of this section $H$, always denotes a finite reflexive graph which is locally triangulated around a vertex $0$. We define a directed graph analogously to Definition~\ref{PhiQuad}, guided by the graphs in Figure \ref{fig:nbo2} rather than in Figure \ref{fig:nbo}.  

\begin{defn}
Let $A$ be the set of all pairs $(x,y)\in V(H)^2$ such that $x\neq y$, $xy\in E(H)$ and $\{x,y\}$ is listable. 
\end{defn}

\begin{defn}
\label{PhiTri}  
Let $\Phi$ be a directed graph on vertex set $A$ where there is an arc from $(a,b)$ to $(c,d)$ if $ac,bc,bd\in E(H)$ and $ad\notin E(H)$. 
\end{defn}

\begin{figure}[htbp]
\begin{center}
\begin{tikzpicture}
   \newdimen\R
   \R=1.65cm
   \newdimen\smallR
   \smallR=1.0cm
   
   \draw\foreach [count=\n from 1] \x in {1,2,3,m} {
     (1+2*\n,3) node [smallblack, label={[label distance=0pt]90:{$y_{\x}$}}] (u\n){}
   };
   \draw (7.5,3) node (u31){};
   \draw (8.5,3) node (u39){};
   \draw(u1)--(u2)--(u3)--(u31);\draw(8,3) node (){$\cdots$};\draw(u39)--(u4);
   \draw\foreach [count=\n from 1] \x in {1,2,3,4,m} {
     (2*\n,2) node [smallblack, label={[label distance=0pt]-45:{$d_{\x}$}}] (d\n){}
     (2*\n,0) node [smallblack, label={[label distance=0pt]135:{$c_{\x}$}}] (c\n){}
     (c\n)--(d\n)
   };
   \foreach [count=\n from 1] \x in {2,3,4} {
     \draw(c\n)--(c\x);
     \draw(d\n)--(d\x);
     \draw(c\n)--(d\x); };
   \draw\foreach \x in {0,1,2} {
     (9,\x) node (){$\cdots$}
   };
   \draw (8.5,0) node (c42){};
   \draw (9.5,0) node (c48){};
   \draw (8.5,2) node (d42){};
   \draw (9.5,2) node (d48){};
   \draw (c4)--(c42); \draw (c48)--(c5);
   \draw (d4)--(d42); \draw (d48)--(d5);
   \draw(1,1) node (){$L$};
   \draw(2,4) node (){$P$};
\end{tikzpicture}
\end{center}

\caption{Graphs $P$ and $L$ motivating $\Phi$ of Defintion \ref{PhiTri}.}
\label{fig:nbo2}
\end{figure}
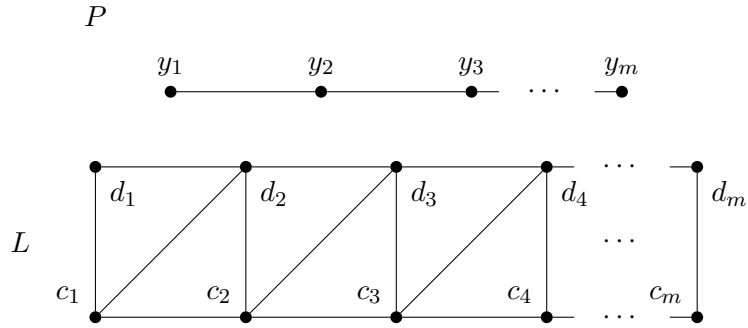

We remark that an important difference between the digraph $\Phi$ defined in this section and the one for irreflexive quadrangulations in Section~\ref{quadSection} is that, for given an arc $(a,b)(c,d)$, the vertices $b$ and $c$ may actually coincide. This is the key property which allows us to get away with local arguments in this section. 

\begin{obs}
\label{reverseTri}
There is an arc from $(a,b)$ to $(c,d)$ in $\Phi$ if and only if there is an arc from $(d,c)$ to $(b,a)$. 
\end{obs}

The following lemma is proved in a way which is analogous to the proof of Lemma~\ref{pathImpliesGadget}. We omit the  details. 

\begin{lem}
\label{pathGadgetTri}
If there is a directed path from $(a_1,a_0)$ to $(b_0,b_1)$ in $\Phi$, then there exists a reflexive $\left\{(a_0,b_0),(a_1,b_0),(a_0,b_1)\right\}$-gadget. 
\end{lem}

Thus, Lemma~\ref{nboLemLocal} is implied by the following lemma, via Lemma~\ref{pathGadgetTri}. 

\begin{lem}
\label{nboTri2}
For $1\leq i\leq k$, there is a path from $(a_1,a_0)$ to $(b_0,b_1)$ in $\Phi$ for the following choices of $a_0,a_1,b_0,b_1$
\begin{enumerate}[(i)]
\item\label{11} $a_0=b_0=0$ and $a_1=b_1=1$,
\item\label{0i} $a_0=1$, $a_1=0$, $b_0=\beta_i$ and $b_1=i$ and
\item\label{0i+1} $a_0=1$, $a_1=0$, $b_0=\beta_i$ and $b_1=i+1$.
\end{enumerate}
\end{lem}

\begin{proof}
Since the subgraph of $H$ induced by $V(F)$ is $K_4$-free, we know that vertex $i$ is not adjacent to $i+2$ for $1\leq i\leq k$. Therefore, for \eqref{11}, we simply observe that the following is a directed path in $\Phi$:
\[(1,0)(2,3)(3,4)\cdots (k-1,k)(0,1).\]
For $1\leq j\leq k$, label the neighbours of $\beta_j$ by $x_1^j,\dots,x_{t_j}^j$ so that $x_1^j=j$ and $x_1^j\cdots x_{t_j}^j$ is a cycle. Again, since the subgraph of $H$ induced by $V(F)$ is $K_4$-free, we know that $x_{\ell}^j$ and $x_{\ell+2}^j$ are non-adjacent. Also, $0$ and $\beta_j$ are non-adjacent. Thus, the following is a directed path in $\Phi$:
\[(0,1)(2,\beta_1)(x_{2}^1,x_3^1)\cdots (x_{t_1-1}^1,x_{t_1}^1)(\beta_1,2)(1,0).\]
We follow this by the directed path
\[(1,0)(2,3)(3,4)\cdots (i-2,i-1)(0,i)(i,\beta_i).\]
We complete the proof of \eqref{0i} by continuing along the following path
\[(i,\beta_i)(x_2^i,x_3^i)\cdots (x_{t_i-1}^i,x_{t_i}^i) (\beta_i,i).\]
The argument used to prove \eqref{0i} applies \emph{mutatis mutandis} to prove \eqref{0i+1}. 
\end{proof}

\subsection{Not-All-Zero Gadget for Locally Triangulated Graphs}

Our next goal is to prove Lemma~\ref{nazLemLocal}, which, when combined with Lemma~\ref{nboLemLocal} \eqref{11}, will complete the proof of Theorem~\ref{localThm} via Lemma~\ref{gadgetImpliesHard}. 

\begin{proof}[Proof of Lemma~\ref{nazLemLocal}]
Our goal is to construct a not-all-zero gadget $Z(z_1,z_2,z_3,z_4)$. Recall that the neighbours of the vertex $0$ in $H$ are labelled $1,\dots,k$ for some $k \geq 4$ where consecutive neighbours modulo $k$ are adjacent. In what follows, we will often refer to vertex $5$ which, if $k=4$, is regarded as the same as vertex $1$.

As a first step, we apply Lemmas~\ref{nboLemLocal} and~\ref{pathGadgetTri} to disjointly add 
\begin{itemize}
\item a reflexive $\{(1,\beta_i),(0,\beta_i),(1,i)\}$-gadget $W_i(z_i,w_i)$ for $1\leq i\leq 3$ and
\item a reflexive $\{(1,\beta_4),(0,\beta_4),(1,5)\}$-gadget $W_4(z_4,w_4)$.
\end{itemize}
Then, disjointly from the construction so far and from one another, use conditions \eqref{edgeList} and \eqref{triList} of Definition~\ref{localDef} to add 
\begin{itemize}
\item a reflexive $\{0,1\}$-gadget with signal vertex $y_1$,
\item a reflexive$\{0,2,3\}$-gadget with signal vertex $y_2$,
\item a reflexive $\{0,3,4\}$-gadget with signal vertex $y_3$ and
\item a reflexive $\{0,5\}$-gadget with signal vertex $y_4$.
\end{itemize}
Finally, add an edge from $y_i$ to $w_i$ for $1\leq i\leq 4$ and add edges $y_1y_2$, $y_2y_3$ and $y_3y_4$. See Figure~\ref{zPathFig}. 

\begin{figure}[htbp]
\begin{tikzpicture}[scale=1.3]
\draw (0,0) node [smallblack, label={[label distance=27pt]90:{$y_1$}}] (z1) {};
\draw (0,0) node [label={[label distance=7pt]90:{$\{0,1\}$}}] (L1) {};
\draw (2,0) node [smallblack, label={[label distance=27pt]90:{$y_2$}}] (z2) {};
\draw (2,0) node [label={[label distance=7pt]90:{$\{0,2,3\}$}}] (L2) {};
\draw (4,0) node [smallblack, label={[label distance=27pt]90:{$y_3$}}] (z3) {};
\draw (4,0) node [label={[label distance=7pt]90:{$\{0,3,4\}$}}] (L3) {};
\draw (6,0) node [smallblack, label={[label distance=27pt]90:{$y_4$}}] (z4) {};
\draw (6,0) node [label={[label distance=7pt]90:{$\{0,5\}$}}] (L4) {};

\draw (z1) -- (z2);
\draw (z3) -- (z2);
\draw (z3) -- (z4);

\path (z1) edge[ out=69, in=111
                , loop
                , distance=0.4cm]
            node[above=3pt] {} (z1);

\path (z2) edge[ out=69, in=111
                , loop
                , distance=0.4cm]
            node[above=3pt] {} (z2);

\path (z3) edge[ out=69, in=111
                , loop
                , distance=0.4cm]
            node[above=3pt] {} (z3);

\path (z4) edge[ out=69, in=111
                , loop
                , distance=0.4cm]
            node[above=3pt] {} (z4);

\begin{scope}[shift={(0,-1)}]

\draw (0,0) node [smallblack, label={[label distance=0pt]0:{$w_1$}}] (y1) {};
\draw (0,0) node [label={[label distance=7pt]180:{$\{\beta_1,1\}$}}] (L1) {};
\draw (2,0) node [smallblack, label={[label distance=0pt]0:{$w_2$}}] (y2) {};
\draw (2,0) node [label={[label distance=7pt]180:{$\{\beta_2,2\}$}}] (L2) {};
\draw (4,0) node [smallblack, label={[label distance=0pt]0:{$w_3$}}] (y3) {};
\draw (4,0) node [label={[label distance=7pt]180:{$\{\beta_3,3\}$}}] (L3) {};
\draw (6,0) node [smallblack, label={[label distance=0pt]0:{$w_4$}}] (y4) {};
\draw (6,0) node [label={[label distance=7pt]180:{$\{\beta_4,5\}$}}] (L4) {};

\draw (y1) -- (z1);
\draw (y2) -- (z2);
\draw (y3) -- (z3);
\draw (y4) -- (z4);

\path (y1) edge[ out=69+90, in=111+90
                , loop
                , distance=0.4cm]
            node[above=3pt] {} (y1);

\path (y2) edge[ out=69+90, in=111+90
                , loop
                , distance=0.4cm]
            node[above=3pt] {} (y2);

\path (y3) edge[ out=69+90, in=111+90
                , loop
                , distance=0.4cm]
            node[above=3pt] {} (y3);

\path (y4) edge[ out=69+90, in=111+90
                , loop
                , distance=0.4cm]
            node[above=3pt] {} (y4);

\end{scope}

\begin{scope}[shift={(0,-2)}]

\draw (0,0) node [smallblack, label={[label distance=27pt]270:{$z_1$}}] (x1) {};
\draw (0,0) node [label={[label distance=7pt]270:{$\{0,1\}$}}] (L1) {};
\draw (2,0) node [smallblack, label={[label distance=27pt]270:{$z_2$}}] (x2) {};
\draw (2,0) node [label={[label distance=7pt]270:{$\{0,1\}$}}] (L2) {};
\draw (4,0) node [smallblack, label={[label distance=27pt]270:{$z_3$}}] (x3) {};
\draw (4,0) node [label={[label distance=7pt]270:{$\{0,1\}$}}] (L3) {};
\draw (6,0) node [smallblack, label={[label distance=27pt]270:{$z_4$}}] (x4) {};
\draw (6,0) node [label={[label distance=7pt]270:{$\{0,1\}$}}] (L4) {};

\draw[dashed] (x1) -- (y1);
\draw[dashed] (x2) -- (y2);
\draw[dashed] (x3) -- (y3);
\draw[dashed] (x4) -- (y4);

\path (x1) edge[ out=69+180, in=111+180
                , loop
                , distance=0.4cm]
            node[above=3pt] {} (x1);

\path (x2) edge[ out=69+180, in=111+180
                , loop
                , distance=0.4cm]
            node[above=3pt] {} (x2);

\path (x3) edge[ out=69+180, in=111+180
                , loop
                , distance=0.4cm]
            node[above=3pt] {} (x3);

\path (x4) edge[ out=69+180, in=111+180
                , loop
                , distance=0.4cm]
            node[above=3pt] {} (x4);

\end{scope}
\end{tikzpicture}

\caption{The vertices $z_i,w_i$ and $y_i$ for $1\leq i\leq 4$ with their lists. Solid lines represent edges of $Z$ and dashed lines represent gadgets which force $w_i$ to map to $\beta_i$ when $z_i$ maps to $0$.}
\label{zPathFig}
\end{figure}
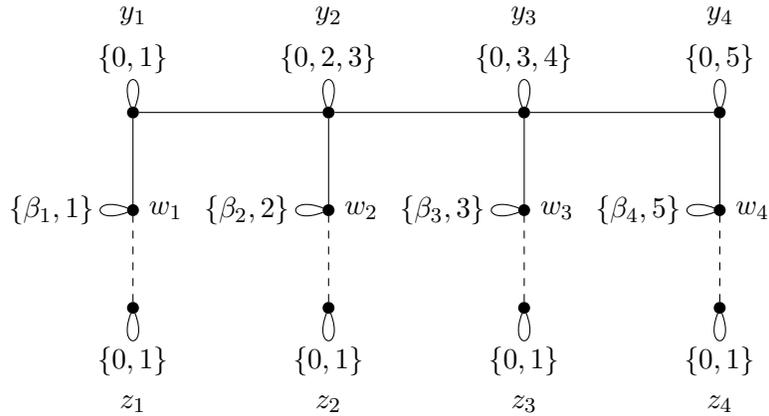

Now, for $p=(p_1,p_2,p_3,p_4)\in\{0,1\}^4\setminus\{(0,0,0,0)\}$, we define $\zeta_p$ as follows. If $p_i=1$ for $1\leq i\leq 3$, then  we colour $W_i(z_i,w_i)$ with $\zeta_{(1,i)}$ and $y_i$ with $0$. Similarly, if $p_4=1$, colour  $W_4(z_4,w_4)$ with $\zeta_{(1,5)}$ and $y_4$ with $0$. On the other hand, if $p_i=0$, then we colour $W_i(z_i,w_i)$ with $\zeta_{(0,\beta_i)}$. If there exists $j>i$ with $p_j=1$, then we colour $y_i$ with the smallest non-zero colour in its list and, otherwise, colour it with the largest such colour.

Let us now show that if $\psi$ reconfigures to $\zeta_p$ for some $p\in\{0,1\}^4\setminus\{(0,0,0,0)\}$, then at least one of $\psi(z_1),\dots,\psi(z_4)$ is one. The gadgets $W_1,\dots,W_4$ imply that, if $\psi(z_1)=0$, then $\psi(w_1)=\beta_1$ which implies $\psi(y_1)=1$ and if $\psi(z_2)=0$, then $\psi(y_2)\in\{2,3\}$ which means that it must be equal to $2$ because $13\notin E(H)$ since $V(F)$ induces a $K_4$-free subgraph of $H$. By the same argument, $\psi(y_4)$ must be $5$ and $\psi(y_3)$ must be $4$, which is a contradiction because $24\notin E(H)$ since $V(F)$ induces a $K_4$-free subgraph of $H$.

Now, suppose that $p,q\in\{0,1\}^4\setminus\{(0,0,0,0)\}$ differ on exactly one coordinate. Without loss of generality, there is some $i$ such that $p_i=0$ and $q_i=1$. Let $j$ be any coordinate such that $p_j=q_j=1$. Starting with $\zeta_p$, we can reconfigure the colouring on the gadget $W_i(z_i,w_i)$ so that $z_i$ maps to $1$ and $w_i$ does not map to $\beta_i$. We can then change the colour of $y_i$ to $0$. After doing this, we change the colours on the vertices $y_{i'}$ for $i'$ between $j$ and $i$ so that they match their colours under $\zeta_q$ (and we can do that without modifying the colourings on the non-signal vertices of the gadgets  $W_{i'}(z_{i'},w_{i'})$). This completes the proof.
\end{proof}

\begin{ack}
We would like to thank an anonymous referee who pointed out that our proof of Theorem~\ref{ppCor}, which we originally only proved for odd wheels, holds for all non-bipartite quadrangulations of the projective plane.  The second author would also like to thank Nima Hoda for several enlightening discussions on topics related to those covered in this paper.
\end{ack}

\bibliographystyle{plain}
\bibliography{sphere}
\end{document}